\documentclass[11 pt, psamsfonts]{amsart}
\usepackage{fullpage} % Package to use full page
\usepackage{parskip} % Package to tweak paragraph skipping
\usepackage{tikz} % Package for drawing
\usepackage{amsmath,amssymb,amsfonts,amsthm}
\usepackage{hyperref}
\usepackage{longtable}
\usepackage{indentfirst}
\usepackage{xcolor}
\theoremstyle{definition}
\usepackage{array} % for extrarowheight
\usepackage{times}
\usepackage[T1]{fontenc}

\usepackage{graphicx}
\newtheorem{thm}{Theorem}[section]
\theoremstyle{definition}

\DeclareGraphicsExtensions{.png}

\DeclareMathOperator{\Rep}{Rep}
\DeclareMathOperator{\Fib}{Fib}
\DeclareMathOperator{\Sem}{Sem}

\DeclareMathOperator{\Id}{Id}
\DeclareMathOperator{\SU}{SU}

\DeclareMathOperator{\SO}{SO}
\DeclareMathOperator{\SL}{SL}

\DeclareMathOperator{\PSU}{PSU}

\DeclareMathOperator{\sVec}{sVec}
\DeclareMathOperator{\Gal}{Gal}
\DeclareMathOperator{\Hom}{Hom}

\DeclareMathOperator{\1}{\textbf{1}}

\DeclareMathOperator{\FPdim}{FPdim}
\newcommand{\er}[1]{\color{blue}#1\color{black}}
\newcommand{\jp}[1]{\color{cyan}#1 -jp\color{black}}
\newcommand{\qz}[1]{\color{magenta}#1 -qz\color{black}}
\newcommand{\paren}[1]{\left(#1\right)}
\newcommand{\CC}{\mathcal{C}}
\newcommand{\DD}{\mathcal{D}}

\newcommand{\FF}{\mathcal{F}}
\newcommand{\BB}{\mathcal{B}}
\newcommand{\T}{\mathcal{T}}

\newcommand{\bbQ}{\mathbb{Q}}
\newcommand{\hlam}{\hat{\lambda}}
\newcommand{\hT}{\hat{T}}
\newcommand{\hS}{\hat{S}}
\newcommand{\hN}{\hat{N}}
\newcommand{\tS}{\tilde{S}}
\newcommand{\ot}{\otimes}

\newcommand{\one}{\mathbf{1}}

\newcommand{\Z}{\mathbb{Z}}

\newtheorem{theorem}[thm]{Theorem}
\newtheorem{conj}{Conjecture}[section]
\newtheorem{lem}{Lemma}[section]
\newtheorem{Prop}{Proposition}[section]
\newtheorem{example}{Example}[section]

\newtheorem{rmk}{Remark}
\newtheorem{Def}{Definition}[section]
\newtheorem{Cor}{Corollary}[section]
\newtheorem{Question}{Question}

\title{Classification of super-modular categories}
\author{Paul Bruillard$^1$,  Julia Plavnik$^2$, Eric C. Rowell$^3$, Qing Zhang$^3$}
\date{\today}

\address{$^1$Expedia Group \\
Bellevue, WA 98004\\
U.S.A.}
\email{Paul.Bruillard@pnnl.gov}
\address{$^2$Department of Mathematics\\
    Indiana University\\
    Bloomington, IN 47405\\
    U.S.A.}
\email{jplavnik@iu.edu}
\address{$^3$Department of Mathematics\\
    Texas A\&M University\\
    College Station, TX 77843-3368\\
    U.S.A.}
\email{rowell@math.tamu.edu, zhangqing@math.tamu.edu}

\thanks{ECR and QZ were partially supported by US NSF grant MPS-1664359, and a Presidential Impact Fellowship of Texas A\&M. JP was partially supported by NSF grants DMS-1802503 and DMS-1917319. ECR gratefully acknowledges the support of the Simons Foundation through a Simons Fellowship.
The authors thank M. Papanikolas  for helpful discussions.}

\begin{document}
\begin{abstract}
We develop categorical and number theoretical tools for the classification of super-modular categories.  We apply these tools to obtain a partial classification of super-modular categories of rank $8$.  In particular we find three distinct families of prime categories in rank $8$ in contrast to the lower rank cases for which there is only one such family.
\end{abstract}
\maketitle

\section{Introduction}

The classification of braided fusion categories (BFCs) stands as a formidable, yet enticing problem.  There are many approaches to this problem, with varying levels of preciseness and corresponding degrees of difficulty--as examples, one might try to classify by categorical dimension \cite{MR2098028, naidu2011finiteness, bruillard2016classification, bruillard2017categorical, bruillard2017categorical, bruillard2013classification, yu}, by Witt class \cite{DMNO,davydov2013structure}, by dimension of a generating object \cite{AMP,EM1, EM2}, or by rank \cite{RSW,ostrik2002fusion}.  Each of these approaches have different motivations and have seen some measure of success.  For example classifying by categorical dimension is related to the problem of classifying groups  by their orders, while classifying by the dimension of a generating object is related to the classification of finite index, finite depth subfactors.  Classification by rank can be motivated physically: for condensed matter systems (e.g. topological phases of matter) modeled by braided fusion categories, the rank of the category corresponds to the number of distinguishable indecomposable particle species \cite{Nayaketal}.  In this article we will be interested in classification by (low) rank of \emph{unitary} BFCs, as motivated by this physical interpretation.

Interestingly, the classification of low rank fusion categories has not progressed very far: it is an open question as to whether there are finitely many fusion categories of each rank, whereas with the braiding assumption rank-finiteness is known \cite{BNRW1,jones2019rank}.  The classification of pivotal fusion categories is complete up to rank 3, while the braiding assumption allows one to go a bit further, for example, there is a complete classification up to rank 5 of pre-modular fusion categories \cite{MR3548123}, \cite{BOM}. One reason is as follows, which also serves to motivate this paper more specifically:
It is well-known \cite{DGNO} that if $\BB$ is a braided fusion category and $\Rep(G)\cong \BB^\prime_{\T}\subset \BB$ is the maximal Tannakian subcategory of the M\"uger center $\BB^\prime$ of $\BB$, then the $G$-de-equivariantization $\BB_G$ of $\BB$ is either non-degenenerate (has trivial M\"uger center) or slightly degenerate (has M\"uger center equivalent to $\sVec$).  For unitary BFCs this produces either a unitary modular tensor category (in the non-degenerate case) or a super-modular category (in the slightly degenerate case).  Thus, if one is interested in unitary braided fusion categories ``modulo finite group representations" one is led to study modular or super-modular categories.

Techniques for classifying modular categories are well-established (for example, see \cite{RSW,BNRW2}), and the classification up to rank $6$ is nearly complete \cite{Creamer, Green}.  Those methods cannot always be applied to general braided fusion categories.  For example, a key approach in \cite{BNRW2} is to use the representation theory of the modular group $\SL(2,\Z)$ to put constraints on the (modular) $S$- and (twist) $T$-matrices, whereas a super-modular category doesn't not provide such representations, as the $S$ matrix has determinant $0$.  On the other hand, there is an important conjecture known as the \emph{minimal modular extension} (MME) conjecture \cite{DMNO,16fold} that predicts that any super-modular category $\BB$ can be embedded in a modular category $\CC$ with $\dim(\CC)=2\dim(\BB)$.  Necessarily such a $\CC$ will be a \emph{spin modular} category, i.e.\ a modular category with a distinguished fermion $f$, and $\BB=\langle f\rangle^\prime$ is the M\"uger centralizer of the category generated by $f$.  

\iffalse
%Two key facts are as follows:

%\begin{itemize}
 %   \item If a minimal modular extension $\CC$ of $\BB$ exists, there are precisely 16 inequivalent such extensions \cite{KLW,16fold}.
 %   \item $3/2\rank{\BB}\leq\rank{\CC}\leq 2\rank{\BB}$.
%\end{itemize}
We will  see that a classification of super-modular categories of rank$\leq 8$ implies a classification of rank$\leq 14$ spin modular categories, since the largest rank that $\langle f\rangle^\prime$ could have is $8$.  Similarly, if we assumed the MME conjecture and classified spin modular categories of rank $16$ then we could obtain a complete classification of rank$\leq 8$ super-modular categories.
\fi

Some techniques for classifying super-modular categories have been developed recently \cite{16fold,bruillard2017classification}, which lead to a complete classification up to rank $6$.  It turns out that there are really very few such categories: modulo trivial Deligne product constructions and up to fusion rules there are only two examples with rank $\leq 6$, and both of them belong to the a family of super-modular categories arising from quantum groups.  A particularly useful technique is to formally ``condense the fermion" to obtain a fermionic quotient, which has naive fusion rules.  These can be studied using the concept of a $\sVec$-enriched fusion category \cite{Usher, MP}, but we will not pursue that here. In this article we make some partial progress towards the classification of rank $8$, using a stratification by Galois group and some new techniques.  We find that there are many non-trivial examples, in contrast to lower ranks, and we were unable to give a definitively complete classification.  

For the following the (standard) notation is explained in the appendix.
\begin{theorem}
\begin{enumerate}
    \item 

The following are constructions of prime rank $8$ super-modular categories as centralizers of a distinguished fermion in spin modular categories:
\begin{enumerate}
    \item $\PSU(2)_{14}=\langle f\rangle^\prime\subset SU(2)_{14}$ where $f$ is the unique fermion corresponding to highest weight $7\varpi$.
    \item $[\PSU(2)_6\boxtimes\PSU(2)_6]_{\Z_2}=\langle \overline{(f,\one)}\rangle^\prime\subset ([\SU(2)_6\boxtimes\SU(2)_6]_{\Z_2})_0$ where the $\Z_2$-de-equivariant{-}ization in both cases is with respect to the boson $(f,f)$ where $f$ has highest weight $3\varpi$, and $\overline{(f,\one)}$ is the image of $(f,\one)$ under de-equivariantization.
    \item $\langle f\rangle^\prime\subset SO(12)_2$, where $f$ is either of the fermions labelled by $2\varpi_{5}$ or $2\varpi_6$.
\end{enumerate}
\item Moreover, if we assume that the naive fusion rules $\{\hN_{ij}^k=N_{ij}^k+N_{ij}^{fk}\}_{i,j,k}$ %are bounded by $14$ 
and the simple objects' dimensions $d_i$ are each bounded by $14$, then any prime super-modular category of rank $8$ has the same fusion rules as one of the above.\end{enumerate}
\end{theorem}

A more precise classification with less stringent bounds can be found in Section \ref{main section}.

While we cannot claim this is a complete classification as we have placed bounds in some cases on naive fusion rule multiplicities or dimensions, it is possible that we have listed all possibilities. 
A counterexample would have large naive fusion multiplicities/dimensions compared to the known examples: the largest naive fusion multiplicity we find among fermionic quotients is 4  while the largest dimension of a simple object is  $3+2\sqrt{2}\approx 5.8$. There is some precedent for these types of constraints: \cite{gepnerkapsutin} gives a classification of low rank modular categories with bounded fusion multiplicities and \cite{XGWen1506.05768} uses numerical techniques to study low rank modular categories with constrained categorical dimension.  Although our result is not complete, we provide some new powerful methods for classifying super-modular categories, and illustrate the utility of the existing techniques.

\iffalse
\begin{itemize}
    \item Explain why this problem is important
    \item Explain what is already know about super-modular
    \item explain what are the complications/differences with respect to the modular case.
    \item explain notation: $\PSU(N)_k$, $\Fib$, $\Sem$ etc.
    \item might be worth mentioning: XG Wen wrote a paper with bounded $D^2$ (arxiv: 1506.05768) for modular up to rank 9 or so.
\end{itemize}
\fi
In this paper we assume that the reader is familiar with the notions and basic properties of fusion, braided and modular tensor categories. For details, we refer to \cite{MR3242743, bakalov2001lectures}.  We provide some of the most relevant details and derive some general results in Section \ref{prelim}.  In Section \ref{main section} we state our main results in detail and complete the first step of our classification, which determine the naive fusion rules.  In Section \ref{fusion rules} we lift the naive fusion rules to those of super-modular categories.  In the Appendix we explain some of the notation and give $S$- and $T$-matrices for a realization of each prime super-modular category of rank $8$.
 
\section{Preliminaries}\label{prelim}

In this section, we first introduce the notion of super-modular categories and some of its properties. Most of the results can be found in (\cite{16fold, bruillard2017classification}) and the references therein. Then we discuss the Galois symmetry for super-modular categories.
\subsection{Centralizers}
Whereas one may always define an $S$-matrix for any ribbon fusion category $\BB$, it may be degenerate.  This failure of modularity is encoded it the subcategory of transparent objects called the \textbf{M\"uger center} $\BB^\prime$. Here an object $X$ is called \textbf{transparent} if all the double braidings with $X$ are trivial: $c_{Y,X}c_{X,Y}=\Id_{X\ot Y}$ for all $Y\in\BB$. Generally, we have the following notion of the centralizer of the braiding.

\begin{Def}
	The \textbf{M\"uger centralizer} of a subcategory $\DD$ of a pre-modular category $\BB$ is the full fusion subcategory 
	$$\DD^\prime=C_{\BB}(\DD)=\{X\in \BB|c_{Y,X}c_{Y,X}=\Id_{X\otimes Y},\forall Y\in\DD\}.$$
	The \textbf{M\"uger center} of $\BB$ is the centralizer $\BB'$ of $\BB$ itself, that is, $\BB'=C_\BB(\BB)$.
\end{Def}
While the notation $\DD^\prime$ is slightly ambiguous as it is relative to an ambient category, the context will always make it clear.

By a theorem of Brugui\`eres \cite{Brug}, the simple objects in $\BB^\prime$ are those $X$ with $\tS_{X,Y}=d_Xd_Y$ for all simple $Y$, where $d_Y=\dim(Y)=\tS_{\one,Y}$ is the categorical dimension of the object $Y$.  The M\"uger center is \textbf{symmetric}, that is, $c_{Y,X}c_{X,Y}=\Id_{X\ot Y}$ for all $X,Y\in\BB^\prime$. Symmetric fusion categories have been classified by Deligne in terms of representations of supergroups  \cite{Del}. In the case that $\BB^\prime\cong\Rep(G)$ (i.e., $\BB^\prime$ is \textbf{Tannakian}), the de-equivariantization procedure of Brugui\`eres \cite{Brug} and M\"uger \cite{M3} yields a modular category $\BB_G$ of dimension $\dim(\BB)/|G|$.  Otherwise, by taking a maximal Tannakian subcategory $\Rep(G)\subset \BB^\prime$, the de-equivariantization $\BB_G$ has M\"uger center $(\BB_G)^\prime\cong \sVec$, the symmetric fusion category of super-vector spaces. Generally, a braided fusion category $\BB$ with $\BB^\prime\cong\sVec$ as symmetric fusion categories is called \textbf{slightly degenerate} \cite{DGNO}, while if $\BB^\prime\cong\text{Vec}$, $\BB$ is \textbf{non-degenerate.}

The symmetric fusion category $\sVec$ has a unique spherical structure compatible with unitarity and has $S$- and $T$-matrices:
$S_{\sVec}=
\frac{1}{\sqrt{2}}\begin{pmatrix}
1&1\\1&1\\
\end{pmatrix}$ and $T_{\sVec}=
\begin{pmatrix}
1&0\\0&-1\\
\end{pmatrix}$.

From this point on we will assume that all our categories are unitary, so that $\sVec$ is a unitary spherical symmetric fusion category and all categorical dimensions are equal to the largest eigenvalue of the corresponding fusion matrix, i.e., the Frobenius-Perron dimension.  In particular for any simple object $X$, $d_X\geq 1$.

\subsection{Definition of a super-modular category}
\begin{Def}
	A unitary pre-modular category $\BB$ is called \textbf{super-modular} if $\BB'\simeq \sVec$.
\end{Def}

\begin{rmk}
	In other terminology, we say $\BB$ is super-modular if its M\"uger center is generated by a \textbf{fermion}, that is, an object $f$ with $f^{\otimes 2}\cong\one$ and $\theta_f=-1$. We restrict to unitary categories both for mathematical convenience and for their physical significance.  On the other hand, there is a non-unitary version $\sVec^{-}$ of $\sVec$: the underlying (non-Tannakian) symmetric fusion category is the same, but with the other possible spherical structure, which leads to negative categorical dimensions.  We could define super-modular categories more generally as pre-modular categories $\BB$ with M\"uger center equivalent to either of $\sVec$ or $\sVec^{-}$.  However, we do not know of any examples $\BB$ with $\BB^{\prime}\cong\sVec^{-}$ that are not simply of the form $\CC\boxtimes\sVec^{-}$ for some modular category $\CC$.
\end{rmk}

Super-modular categories (or slight variations) have been studied from several perspectives, see \cite{Bon,DMNO,davydov2013structure,16fold,BCT,KLW,bruillard2017classification,yu} for a few examples.  An algebraic motivation for studying these categories is the following: any unitary braided fusion category is the equivariantization \cite{DGNO} of either a modular or super-modular category (see \cite[Theorem 2]{Sawin}).  Physically, super-modular categories provide a framework for studying fermionic topological phases of matter \cite{16fold}.  Topological motivations include the study of spin 3-manifold invariants (\cite{Sawin,Bl,BM}) and $(3+1)$-TQFTs (\cite{WW}).

A braided fusion category is called \textbf{prime} if it contains no non-trivial non-degenerate braided fusion subcategories.  Indeed, if $\DD\subset\BB$ with $\DD$ non-degenerate and $\BB$ a braided fusion category then $\BB\cong\DD\boxtimes\DD'$  as braided fusion categories\cite[Theorem 3.13]{DGNO} (see also \cite{Mu03}). 
	As a special case of non-prime categories we say a super-modular category $\CC$ is \textbf{split} if $\CC\simeq \sVec\boxtimes \DD$ for some modular subcategory $\DD\subset \CC$, and otherwise $\CC$ is \textbf{non-split}.

\subsection{Spin Modular Categories}

A \textbf{modular category} $\CC$ is a modular category with a distinguished fermion. 
Let $\CC$ be a spin modular category, with fermion $f$, (unnormalized) $S$-matrix $\tS$ and $T$-matrix $T$. Proposition II.3 of \cite{16fold} provides a number of useful symmetries of $\tS$ and $T$:
\begin{enumerate}
	\item $\tS_{f,\alpha}=\epsilon_\alpha d_\alpha$, where $\epsilon_\alpha=\pm 1$ and $\epsilon_f=1$,
	
	\item $\theta_{f \alpha}=-\epsilon_\alpha \theta_\alpha$,\label{thetasign}
	
	\item $\tS_{f\alpha,\beta}=\epsilon_\beta \tS_{\alpha,\beta}$.\label{smatrixsign} 
\end{enumerate}
\begin{rmk}\label{Z2grading}
	We have a canonical $\Z_2$-grading $\CC_0\oplus \CC_1$ with simple objects $X\in\CC_0$ if $\epsilon_X=1$ and $X\in\CC_1$ when $\epsilon_X=-1$.  The trivial component $\CC_0$ is a super-modular category, since $\CC_0^\prime=\langle f \rangle\cong \sVec$.  
\end{rmk}

\iffalse
Further, we have the canonical decomposition $\CC_1=\CC_v\oplus \CC_\sigma$ as abelian categories, where $X\in\CC_v$ if $X\otimes f\neq X$ and $X\in\CC_\sigma$ otherwise. The following result is useful for classifying spin modular categories.\er{do we use this?}
\begin{lem}\cite[Lemma 4.2]{bruillard2017classification}
	Let $ (\CC, f) $ be a spin modular category with $ \CC_0 $, $ \CC_v $ and $ \CC_\sigma $. Denote their rank as $ |\CC_0| $, $ |\CC_v| $ and $ |\CC_\sigma |$. Then we have
	\begin{enumerate}
		\item $ |\CC_0=|\CC_v|+2|\CC_\sigma|$, in particular $ |\CC|=2|\CC_0|-|\CC_\sigma|$.
		\item $ \dfrac{3|\CC_0|}{2}\leq |\CC|\leq 2|\CC_0|$.
		\item $ |\CC_v| $ and $ |\CC_0| $ are even.
	\end{enumerate}
\end{lem}
\fi
%\begin{example}
%    \textcolor{red}{Quantum groups}
%\end{example}

\begin{Def}
	Let $ \BB $ be a ribbon fusion category. A \textbf{minimal modular extension} MME of $ \BB $ is a modular category $ \CC $ such that $ \BB \subset\CC $ and $ \FPdim(\CC)=\FPdim(\BB')\FPdim(\BB) $.
\end{Def}
It is known that not every ribbon fusion category has a minimal modular extension \cite{GV}.
 Notice that if $ \BB $ is super-modular, a minimal modular extension of $ \BB $ is a spin modular category $ (\CC, f) $, where the fermion $ f $ is transparent in $ \BB $.  It is conjectured (see \cite{DMNO,16fold}) that \emph{every} super-modular category has an MME, and it is known \cite{KLW,16fold} that if one exists there are precisely $16$ inequivalent such extensions.  A complete classification of rank$\leq 8$ super-modular categories would include a classification of rank$\leq 14$ spin modular categories, whereas if the MME conjecture is true a classification of spin modular categories of rank$\leq 16$ would imply a classification of super-modular categories of rank$\leq 8$.

\subsection{Fermionic Quotient}
One interesting feature of super-modular categories $\BB$ is that their $S$ and $T$ matrices have tensor decompositions:
\begin{theorem}\cite[Theorem 3.5]{16fold}\label{decomp}
	Let $\BB$ be a super-modular category, then  $\tilde{S}={\scriptsize \begin{pmatrix} 1 & 1\\1 &1\end{pmatrix}}\otimes\hat{S}$ and $T={\scriptsize \begin{pmatrix}1 &0\\0&-1\end{pmatrix}}\otimes\hat{T}$, with $\hat{S}$ a symmetric invertible matrix and $ \hat{T}$ a diagonal matrix.
\end{theorem}
Recall that for the category $\sVec$, we have $\tS_{\sVec}={\scriptsize \begin{pmatrix} 1 & 1\\1 &1\end{pmatrix}}$ and $T_{\sVec}={\scriptsize \begin{pmatrix}1 &0\\0&-1\end{pmatrix}}$.

\begin{Def}
    $ \hS $ and $ \hT $ are called the \textbf{$ S $- and $ T $-matrix of the fermionic quotient}.
\end{Def}
%Consequently, the Gauss sums of a super-modular category is always 0, i.e., $p^{\pm}:=\sum_{i\in \Pi_{\BB}}\theta_i^{\pm}d_i^2=0$.  \jp{this seems a little bit out of place or if some preliminary is missing. Maybe connect more or put after the fact that $f\otimes$ is fixed point free.}\er{I would say we drop it.}

%If a super-modular $ \BB\cong \CC\boxtimes \sVec $ for some modular $ \CC $, we call it \textbf{split}, otherwise we say it is \textbf{non-split}.\jp{we already gave this definition in Def 2.3}

By the following proposition, pointed super-modular categories always splits.
\begin{Prop}\cite[Corollary A.19.]{DGNO}\label{pointedsplits}
	Let $ \BB $ be a pointed super-modular category, then $ \BB\simeq \CC\boxtimes \sVec $, where $ \CC $ is a pointed modular category.
\end{Prop}

Let $f$ be the transparent fermion in a super-modular category $\BB$ with label set $\Pi_{\BB}$. By the following lemma, we know that $ f\otimes -$ is fixed-point-free on $\Pi_{\BB}$.   We will omit the $ \otimes $ symbol and denote $ f\otimes X $ simply as $ fX $.
\begin{lem}\cite[Lemma 5.4]{MR1749250}
	Let  $\BB$ be a super-modular category with transparent fermion $f$. Then $f X\ncong X$ for any $X\in \Pi_{\BB}$.
\end{lem}
As a direct consequence of the previous lemma, we have that super-modular categories have even rank.

\begin{lem}
Let $\BB$ be a super-modular category with transparent fermion $f$. Then	$fX\not\cong X^{\ast} $ for any $X\in \BB$.
\end{lem}
\begin{proof}
	By the balancing equation (given in by the third equality) we have that
	\begin{align*}
	-\theta_{X}d_{X}&=\theta_{X}\theta_{f}d_{f}d_{X}\\&=\theta_{X}\theta_{f}S_{f,X}=\sum_{Y}N_{f,X}^{Y}d_{Y}\theta_{Y}\\&=d_{fX}\theta_{fX}=d_{X}\theta_{fX}.
	\end{align*}
	Therefore $ \theta_{fX}=-\theta_X $. But since $ \theta_{X^\ast}=\theta_X $, it follows that $fX\not\cong X^{\ast}$.
\end{proof}

Thus there is a non-canonical partition of the label set $\Pi_{\BB} = \Pi_0\sqcup f\Pi_0$. We can arrange this partition such that $ 0\in \Pi_0$ and such that $X^{\ast}\in\Pi_0$ if $X\in \Pi_0$. For a rank $ 2r $
super-modular $ \BB$, we have $ 0, \dots, r-1 \in\Pi_0$ and $f= f 0, \dots, f( r-1) \in f\Pi_0$, where $ f i $ is the label for $ f X_i $, $ i=0,\dots, r-1$.

For $ i,j, k\in\Pi_0 $, we define the \textbf{naive fusion rule}
\[\hN^k_{ij}=\dim\Hom(X_i\otimes X_j, X_k)+\dim\Hom(X_i\otimes X_j, f\otimes X_k)=N_{ij}^k+N_{ij}^{f\dot k}.\]

\begin{Prop}\cite[Proposition 2.7]{bruillard2017classification}\label{prop for super}
	Let $\BB$ be a super-modular category, then 
	\begin{enumerate}
		\item[(a)]  $\hS$ is symmetric and $\hS\bar{\hS}=\frac{D^2}{2}I$.
		\item[(b)] $\hN_i\hN_j=\hN_j\hN_i$ for any $i,j\in \Pi_0$.
		\item[(c)] Let $\{x_i| i\in\Pi_0\}$ denote the basis of the Grothendieck semiring $K_0(\BB)$ of $\BB$. Then the functions $\phi_i(x_j):=\hS_{ij}/\hS_{0i}$ for $0\leq i\leq r-1$ form a set of orthogonal characters of $K_0(\BB)$. Thus $\hS$ simultaneously diagonalizes the matrices $\hN_i$.
		\item[(d)] We have a Verlinde type formula in this context given by  $\hN_{ij}^k=\dfrac{2}{D^2}\sum\limits_{m\in\Pi_0}\dfrac{\hS_{im}\hS_{jm}\bar{\hS}_{km}}{d_m}$.
	\end{enumerate}
\end{Prop}

\begin{Cor}\label{cyclicnaivefusion}
	Let $ \BB $ be a super-modular category and $ \hN_{ij}^k $ be its naive fusion rule, where $i, j, k\in\Pi_0 $. We have the following symmetries
	\[
	\hN_{ij}^k=\hN_{ji}^k=\hN_{ik^\ast}^{j^\ast}=\hN_{i^\ast j^\ast}^{k^\ast}, \qquad \hN_{ij}^0=\delta_{ij^\ast}\]
\end{Cor}
\begin{proof}
	The first equation is a direct consequence of Proposition \ref{prop for super} (d). The second equation can be derived by combining (a) and (d) of Proposition \ref{prop for super}.
\end{proof}

\begin{rmk}
	One can combine Corollary \ref{cyclicnaivefusion} and \cite[Equation 2.4.3]{bakalov2001lectures} to get more relations for the fusion coefficients. For example, we have $ N_{ij}^{f k}=N_{ik^\ast}^{f j^\ast}$. In fact, the result follows from $\hN_{ij}^k=N_{ij}^k+N_{ij}^{f k}= N_{ik^\ast}^{j^\ast}+ N_{ij}^{f k}=N_{ik^\ast}^{j^\ast}+N_{ik^\ast}^{f j^\ast}=\hN_{ik^\ast}^{j^\ast}$. 
\end{rmk}

Mimicking the proof for modular categories (see, e.g., \cite[Lemma 1.2]{MR1617921}), one can derive the following property of the dimensions for super-modular categories.
\begin{Cor}\label{divisibility}\cite[Corollary 3.4]{yu}
	Let $\BB$ be a super-modular category, then $d_i^2|\frac{D^2}{2}$.
\end{Cor}
\begin{proof}
	By Proposition \ref{prop for super}, we know that $\hS\bar{\hS}=\frac{D^2}{2}I$, hence we have \[\frac{D^2}{2}=\sum\limits_{j\in\Pi_0}\hS_{ij}\bar{\hS}_{jk}=\sum\limits_{j\in\Pi_0}\hS_{ij}\hS_{jk^\ast}.\] The second equation comes from the fact that for pre-modular categories, we have $\bar{S}_{ij}=S_{ij^\ast}$ since we can embed them into their Drinfeld center. Therefore, we have $\sum\limits_{j\in\Pi_0}\frac{\hS_{ij}}{d_j}\frac{\hS_{jk^\ast}}{d_j}=\frac{D^2/2}{d_i^2}$. The result follows since the left hand side is an algebraic integer.
\end{proof}
The following property of the second Frobenius-Schur indicator for self-dual objects is useful in Section \ref{hat S computation}.
\begin{lem}\cite[Lemma 2.8.]{bruillard2017classification}\label{FSS}
	Let $ \BB $ be a super-modular category and $ X_i $ a simple object such that $ X_i\cong X_i^\ast $ (i.e. $X_i$ is self-dual), then 
	\[\pm1=\nu_2(X_i)=\dfrac{2}{D^2}\sum_{j,k\in \Pi_0}\hN_{j,k}^id_jd_k\big(\frac{\theta_j}{\theta_k}\big)^2.\]
\end{lem}

\begin{Cor}\label{balancing}(Balancing equation for super-modular categories)
	For a super-modular category of rank $2r$, we have:
	\[\theta_i\theta_j\hS_{ij}=\sum_{k=0}^{r-1}(N_{ij}^k-N_{ij}^{fk})\theta_k d_k.\]
	
\end{Cor}
\begin{proof}
	
	We have 
	\begin{align*}
	\theta_i\theta_j\hS_{ij}&=
	\sum_{k=0}^{2r-1}N_{ij}^k\theta_k d_k\\&=
	\sum_{k=0}^{r-1}N_{ij}^k\theta_k d_k+\sum_{k=r}^{2r-1}N_{ij}^k\theta_k d_k\\&=
	\sum_{k=0}^{r-1}N_{ij}^k\theta_k d_k+\sum_{k=r}^{2r-1}N_{ij}^{f k}\theta_{f k} d_{f k}\\&=
	\sum_{k=0}^{r-1}(N_{ij}^k-N_{ij}^{f k})\theta_k d_k.
	\end{align*}
\end{proof}

\subsection{Galois symmetries for super-modular categories}\label{Galois Symmetries for Super-modular Categories}
In this section we discuss the Galois symmetry in the fermionic quotient of a super-modular category, which is parallel to the modular setting.

Let $ \BB $ be a super modular category and $ \hS $, $ \hT $ and $ \hN_i $ defined as above. We have the following relation for the entries of $ \hS $ and $ \hN_i $ \cite[Equation 2.3]{bruillard2017classification}:
\begin{equation}\label{lam}
\dfrac{\hS_{ij}\hS_{ik}}{\hS_{0,i}}=\sum_{m\in\Pi_0}\hN_{jk}^m\hS_{im}
\end{equation}

This means that $ \hlam_{ij}:=\frac{\hS_{ij}}{\hS_{0j}} $ are eigenvalues of the matrices $ \hN_j $ with eigenvectors $ (\hS_{im})_{m\in\Pi_0} $.  Defining the diagonal matrix $ (\hat{\Lambda}_i)_{jk}=\delta_{jk}\frac{\hS_{ij}}{\hS_{0j}} $, then Equation (\ref{lam}) can be written as $ \hN_i\hS=\hS \hat{\Lambda}_i$ for all $ i\in\Pi_0 $.

\begin{rmk}\label{abelianforsuper}
	Let $\mathbb{Q}(\hS)$ be the smallest field containing all elements of the $S$-matrix. Similarly to the modular setting,  $\mathbb{Q}(\hS)$ is Galois over $\mathbb{Q}$ . Define
	$\Gal(\BB)=\Gal(\mathbb{Q}(\hS)/\mathbb{Q})$. Then $\Gal(\BB)$ is an abelian subgroup  of $\mathfrak{S}_{r}$,  where $ 2r $ is the rank of the corresponding super modular category and $\mathfrak{S}_{r}$ is the symmetric group on $r$ letters.  We
	will use $ \sigma $ for both the element of the Galois group $ \Gal(\BB) $ and its associated element in $\mathfrak{S}_{r}$.
\end{rmk} 

We have
\begin{equation}\label{Galosactionsuper}
\sigma\big(\dfrac{\hS_{ik}}{\hS_{0k}}\big)=\dfrac{\hS_{i\sigma(k)}}{\hS_{0\sigma(k)}}.
\end{equation}

We can also derive a  result parallel to \cite[Equation 2.12]{BNRW2} for the $ S $-matrix of the fermionic quotient:
%Namely,

\begin{Cor}\label{supermodular s symmetry}
	Let $\sigma\in\Gal(\BB)$ and $j,k$ the indices of simple objects in $\Pi_{0}$. Then
	\[\sigma\paren{\hat{S}_{j,k}}=\pm \dfrac{\hS_{j,\sigma{(k)}}}{d_{\sigma{(0)}}}.\]
	Moreover, we have the following symmetries:
\begin{equation}\label{hatSsym}
\hS_{j,k}=\hS_{\sigma(j),\sigma^{-1}(k)}.
\end{equation}
\end{Cor}
\begin{proof}
	By Equation (\ref{Galosactionsuper}), we have
	\begin{align*}
	\sigma\paren{\hat{S}_{j,k}}&=\hat{S}_{j,\sigma\paren{k}}\sigma\paren{d_{k}}/d_{\sigma\paren{k}},\\
	\sigma\paren{d_{k}}&=\hat{S}_{k,\sigma\paren{0}}/d_{\sigma\paren{0}}.
	\end{align*}
	In particular, 
	\begin{align*}
	\sigma\paren{\hat{S}_{j,k}}&=\frac{\hS_{j,\sigma\paren{k}}\hS_{k,\sigma\paren{0}}}{d_{\sigma\paren{0}}d_{\sigma\paren{k}}}.
	\end{align*}
	So it suffices to show that $\dfrac{S_{k,\sigma\paren{0}}}{d_{\sigma\paren{k}}}=\pm1$. The result follows from Lemma \ref{sign} below.
\end{proof}

\begin{Prop}
	Let $\sigma\in\Gal(\BB)$ and view it as a permutation on the index set, then
	$\sigma\paren{k}^{*}=\sigma\paren{k^{*}}$ for all $k$.
\end{Prop}
\begin{proof}
	Let $\tau\in\Gal(\bar{\mathbb{Q}}/\mathbb{Q})$ be complex conjugation. Then we have
	
	\begin{align*}
	\frac{S_{j,k^{*}}}{d_{k^{*}}}&=\overline{(\hS_{j,k}/d_k)}\\
	&=\tau\paren{S_{j,k}/d_{k}}\\
	&=S_{j,\tau\paren{k}}/d_{\tau\paren{k}}.
	\end{align*}
	
	Thus $\tau$ sends the normalized $k$-th column to the $\tau\paren{k}$-th
	column which is also the $k^{*}$-th column.  Since $\Gal\paren{\BB}$
	is abelian, we have 
	$\sigma\paren{k}^{*}=\tau\sigma\paren{k}=\sigma\tau\paren{k}=\sigma\paren{k^{*}}$.
\end{proof}

\begin{Cor}
	$S_{k,\sigma\paren{0}}$ is real.
\end{Cor}
\begin{proof}
	The result follows from the following computation
	\begin{align*}
	\overline{S}_{k,\sigma\paren{0}}=S_{k,\sigma\paren{0}^{*}}=S_{k,\sigma\paren{0^{*}}}=S_{k,\paren{0}}.
	\end{align*}
\end{proof}

\begin{lem}\label{sign}
	$\big|\frac{S_{k,\sigma\paren{0}}}{d_{\sigma\paren{k}}}^{2}\big|=1$.
\end{lem}
\begin{proof}
	First we compute
	\begin{align*}
	\sigma\paren{D^{2}}&=\sum_{j}\sigma\paren{d_{j}}^{2}=\sum_{j}\sigma\paren{d_{j}}\sigma\paren{d_{j}^{*}}\\
	&=\sum_{j}\frac{S_{j,\sigma\paren{0}}}{d_{\sigma\paren{0}}}\frac{S_{j^{*},\sigma\paren{0}}}{d_{\sigma\paren{0}}}\\
	&=\frac{1}{d_{\sigma\paren{0}}^{2}}\sum_{j}S_{j,\sigma\paren{0}}\paren{S_{j,\sigma\paren{0}}}^{*}=\frac{D^{2}}{d_{\sigma\paren{0}}^{2}}.
	\end{align*}
	
	On the other hand, we have
	\begin{align*}
	\sigma\paren{D^{2}}&=\sum_{j}\sigma\paren{S_{j,k}S_{j,k}^{*}}=\sum_{j}\sigma\paren{S_{j,k}}\sigma\paren{S_{j,k^{*}}}\\
	&=\sum_{j}\paren{\frac{S_{j,\sigma\paren{k}}S_{k}\sigma\paren{0}}{d_{\sigma\paren{0}}d_{\sigma\paren{k}}}}\paren{\frac{S_{j,\sigma\paren{k^{*}}}S_{k^{*},\sigma\paren{0}}}{d_{\sigma\paren{0}}d_{\sigma\paren{k^{*}}}}}\\
	&=\frac{S_{k,\sigma\paren{0}}S_{k^{*},\sigma\paren{0}}}{d_{\sigma\paren{0}}^{2}d_{\sigma\paren{k}}d_{\sigma\paren{k^{*}}}}\sum_{j}S_{j,\sigma\paren{k^{*}}}S_{j,\sigma\paren{k}}\\
	&=\frac{S_{k,\sigma\paren{0}}S_{k^{*},\sigma\paren{0}}}{d_{\sigma\paren{0}}^{2}d_{\sigma\paren{k}}d_{\sigma\paren{k^{*}}}}D^{2}.
	\end{align*}
	
	Since $d_{\sigma\paren{k^{*}}}=d_{\sigma\paren{k}^{*}}=d_{\sigma\paren{k}}$ and
	$S_{k^{*},\sigma\paren{0}}=\overline{S}_{k,\sigma\paren{0}}=S_{k,\sigma\paren{0}}$, the result follows because
	$D^{2}/d_{\sigma\paren{0}}^{2}$ is nonzero.
\end{proof}

Let $(\CC, f)$ be a spin modular category, recall that the fermion $f$ gives a grading $\CC_0\oplus \CC_1$.
\iffalse We partition the basis $\Pi=\Pi_0\bigsqcup f\Pi_0\bigsqcup \Pi_v\bigsqcup f\Pi_v\bigsqcup\Pi_\sigma$ for the Grothendieck ring of $\CC$.\\
We can write the (normalized) s-\footnote{JP: should we comment on the notation for different S-matrices??}\er{maybe we don't need these explicitly.  I think all we use is $\CC_0$ and $\CC_1$.  Also, I thought $\hS$ was unnormalized...} and T- matrix for $(\CC,f)$ in the following way \cite{bonderson2018congruence}, 
$$s=\begin{pmatrix}\frac{1}{2}\hS & \frac{1}{2}\hS & A&A &X\\
\frac{1}{2}\hS & \frac{1}{2}\hS & -A &-A &-X\\
A^T&-A^T & B& -B&0\\
A^T&-A^T &-B & B &0\\
X^T &-X^T &0 &0 &0
\end{pmatrix}\quad T=\begin{pmatrix} \hat{T}&0& 0&0&0\\
0&-\hat{T} & 0 &0 &0\\
0&0 &\hat{T_{v}}& 0&0\\
0&0 &0 & \hT_{v} &0\\
0 &0 &0 &0 &{T_{\sigma}}
\end{pmatrix}.$$
\fi

\begin{lem}\label{extension lemma}
	Let $(\CC, f)$ be spin-modular with (unnormalized) $S$-matrix $S$,  and $\hS$ the S-matrix for the fermionic quotient. Then $[\mathbb{Q}(S):\mathbb{Q}(\hS)]=2^n$, for some $n$.
	%where $n=0,1$ or $2$.
\end{lem}

\begin{proof} Denote by $S^{(0,0)},S^{(0,1)}=[S^{(1,0)}]^T$ and $S^{(1,1)}$ the $2\times 2$ blocks of the  $S$-matrix $S$ relative to the grading $\CC_0\oplus\CC_1$.  Suppose that $X_a,X_b\in \CC_1$ so that $S_{b,a}$ is an entry in $S^{(1,1)}$.  Then, since the normalized $i$th column $S_{i,a}/d_a$ is a character of the Grothendieck ring $K_0(\CC)$ for each $i$, we see that $(S_{b,a})^2=d_a^2\sum_j N_{b,a}^jS_{j,a}/d_a$.  Since $N_{b,a}^j=0$ if $X_j\in \CC_1$ we find that $(S_{b,a})^2$ lies in the field generated by the entries of $S^{(0,1)}$.  In particular, $[\mathbb{Q}(S^{(1,1)}):\mathbb{Q}(S^{(0,1)})]=2^k$ for some $k$, since every entry of $S^{(1,1)}$ satisfies a polynomial equation of degree $\leq 2$ over $S^{(0,1)}$. 

Now let $S_{b,c}$ be an entry of $S^{(0,1)}=[S^{(1,0)}]^T$, i.e. $X_b\in\CC_1$ and $X_c\in\CC_0$.  A similar argument shows that $(S_{b,c})^2$ lies in the field generated by $S^{(0,0)}$, so that $[\mathbb{Q}(S^{(0,1)}):\mathbb{Q}(S^{(0,0)})]=2^\ell$.  Since $\mathbb{Q}(\hS)=\mathbb{Q}(S^{(0,0)})$, the result follows.

	\iffalse
	$\phi_i:j\mapsto\dfrac{S_{i,j}}{S_{0,i}} $ gives rise to a character of the Grothendieck ring $K_0(\CC)$. If $i, j, j'\in \Pi_v\bigsqcup f\Pi_v\bigsqcup\Pi_\sigma$, then $\dfrac{S_{i,j}}{S_{0,i}} \dfrac{S_{i,j'}}{S_{0,i}} =\Sigma_k N_{j,j'}^k \dfrac{S_{j,j'}}{S_{0,i}}$, where $N_{j,j'}^k =0$ if $j, j'\in \Pi_v\bigsqcup f\Pi_v\bigsqcup\Pi_\sigma$. So $\dfrac{S_{i,j}}{S_{0,i}} \dfrac{S_{i,j'}}{S_{0,i}} $ lies in $\mathbb{Q}(S_{l,m}) $, where $l\in \Pi_v\bigsqcup f\Pi_v\bigsqcup\Pi_\sigma$ and $m\in \Pi_0\bigsqcup f\Pi_0$. In particular, $(\dfrac{S_{i,j}}{S_{0,i}})^2\in \mathbb{Q}(S_{l,m}) $. Similarly, we can show that for any $\dfrac{S_{l,m}}{S_{0,l}}\dfrac{S_{l,m'}}{S_{0,l}}\in \mathbb{Q}(\hS)$, where $l\in \Pi_v\bigsqcup f\Pi_v\bigsqcup\Pi_\sigma$ and $m, m'\in \Pi_0\bigsqcup f\Pi_0$.
	\fi
\end{proof}

\begin{example}
Consider the Ising modular category with label set $\{\one, \sigma,\psi\}$. It is a spin-modular category with fermion $\psi$. Its $S$-matrix is
$$\frac{1}{2}\begin{pmatrix}1 & \sqrt{2}& 1 \\
\sqrt{2}& 0&-\sqrt{2}\\
1&-\sqrt{2}&1\end{pmatrix}.  $$ The subcategory generated by 1 and $\psi$ is $\sVec$, and we have $[\mathbb{Q}(S):\mathbb{Q}(S_{\sVec})]=2.$
\end{example}

\begin{Question}
Is there a relationship between the Galois group of the $S$-matrix of a braided fusion category $\BB$ and that of its Drinfeld center $\mathcal{Z}(\BB)$?
\end{Question}

The following lemma can probably be generalized to non-self-dual categories, but we will only use it in the self-dual case:

\begin{lem}\label{semionlemma}
Suppose that $\BB$ is a self-dual super-modular category and $z$ is a label in the fermionic quotient such that $d_z=1$ and $\hat{S}_{z,z}\neq 1$.  Then $\BB$ contains a modular pointed subcategory equivalent to $\CC(\Z_2,Q)$ (i.e.\ $\Sem$ or $\overline{\Sem}$).
\end{lem}
\begin{proof}
    The hypothesis immediately implies that $\BB$ contains an invertible, self-dual simple object $Z$.  Since $S_{Z,Z}=\hat{S}_{z,z}\neq 1$, the object $Z$ is not self-centralizing, hence generates a modular subcategory of dimension $2$.
\end{proof}

\begin{Question}
Can we drop the self-duality condition in the above, with the same conclusion?
\end{Question}
\subsection{Rank finiteness}
The rank-finiteness property can be extended to categories that do not necessarily admit a spherical structure. It was recently proved that rank-finiteness holds for G-crossed braided fusion categories.
\begin{theorem}\cite[Corollary 4.7.]{jones2019rank}
	There are finitely many equivalence classes of $ G $-crossed braided fusion categories of any given rank.
\end{theorem}

This motives us to pursue a classification of low-rank super-modular categories parallel to \cite{RSW, BNRW2}. A classification of super-modular categories of rank $ \leq 6 $ is given in \cite{bruillard2017classification}. It is shown, for example, that the fusion rules of any non-split super-modular category of rank $\leq 6$ are the same as $ \PSU(2)_{4k+2} $ for $ k = 0, 1 $ and $ 2 $.

\section{Classification of super-modular categories by rank}\label{main section}

\subsection{Main results}

Similarly to modular categories, the Galois group $\Gal (\mathcal B)$  of a super-modular category $\mathcal B$ defined in Section \ref{Galois Symmetries for Super-modular Categories} is an abelian subgroup of the symmetric group $ \mathfrak{S}_{r}$, where $2r$ is the rank of $\mathcal B$ (see Remark \ref{abelianforsuper}).

In this section, we consider the problem of classifying rank $ 2r=8 $ super-modular categories. If $ \BB $ is non-self dual, we can denote the four simple objects in $ \Pi_0 $ as $ \one, Y,X, X^\ast $. The naive fusion rules satisfy the relations in Corollary \ref{cyclicnaivefusion} and the argument in \cite[Appendix A.2]{RSW} works for this case. Therefore, we sometimes assume the super-modular categories are self-dual, in which case $\hS$ has real entries. 

The  abelian subgroups (up to relabeling, but with $0$ distinguished) $ G $ of $ \mathfrak{S}_4 $ are listed in the following table:
\begin{table}[!hbp]\caption{Abelian subgroups of $ \mathfrak{S}_4 $}\label{table2}
	\centering
	\begin{tabular}{|c|c|}
		\hline
		$\langle \textbf{1}\rangle$&$\langle(0)\rangle $\\
		\hline
		$ \mathbb{Z}_2  $ &  $ \langle(01)\rangle$, $ \langle (23)\rangle $, $ \langle(01)(23)\rangle $\\
		\hline
		$  \mathbb{Z}_2\times\mathbb{Z}_2$& $ \langle(01)(23),(02)(13)\rangle$, $ \langle(01),(23)\rangle $\\
		\hline
		$\mathbb{Z}_3 $ &$ \langle (012)\rangle$, $ \langle(123)\rangle $\\
		\hline
		$ \mathbb{Z}_4 $ &$ \langle(0123)\rangle$\\
		\hline
	\end{tabular}
\end{table}

In this section we determine the possible $\hS$-matrices for super-modular categories, and then derive the fusion rules in Section \ref{fusion rules}.
We summarize our results into the following. 
\begin{theorem}\label{5.1}
	Suppose $\BB$ is a rank 8 self-dual super-modular category and $G$ is its Galois group as in Table \ref{table2}  then:
	\begin{itemize}
		\item  If $G= \langle(23)\rangle$,  $\langle(01),(23)\rangle$ or $\langle(123)\rangle$, then $\BB$ does not exist.
		\item If $G= \langle(0)\rangle$, then $\BB$ is \textbf{pointed}, i.e., of the form $\CC(\Z_2\times\Z_2,Q)\boxtimes\sVec$.
		\item If $G= \langle(01)\rangle$, then $\BB$ is \textbf{prime} and \textbf{weakly integral} with the same fusion rules as the centralizer of either fermion in $\SO(12)_2$.
		\item If $G= \langle(01)(23),(02)(13)\rangle$, then $\BB$ has the same fusion as $\Fib\boxtimes \PSU(2)_6$.
	
		\item If $G=\langle(0123)\rangle$ and $\hN_{ij}^k<14$, then $\BB$ is \textbf{prime} and has the same fusion rules as $\PSU(2)_{14}$.
		\item If $G=\langle(012)\rangle$ and $\hN_{ij}^k<21$, then $\BB$ has the same fusion rules as $\PSU(2)_7\boxtimes\sVec$.

\item	If $G= \langle(01)(23)\rangle$ and $ d_i\leq 14 $ for all $ i $, then the fusion rules of $\BB$ are the same as $[\PSU(2)_6 \boxtimes\PSU(2)_6]_{\Z_2}$ and is \textbf{prime},  $\Fib\boxtimes\Fib\boxtimes\sVec$, $\Sem\boxtimes\Fib\boxtimes\sVec$ or $\Sem\boxtimes\PSU(2)_6$.
\end{itemize}
\end{theorem}
  In several cases the proofs in \cite{RSW} for the classification of rank $4$ modular use techniques and results that apply to super-modular categories as well, so we do not repeat the proof here.  For many computations the Gr\"obner basis software in Maple is useful--we used Maple 2018 for our calculations.

\subsection{$\hS $-matrices for rank 8}\label{hat S computation}
The naive fusion coefficients $ \hN_{ij}^k $ can be computed by the entries of $ \hS $ via the Verlinde formula (see Proposition \ref{prop for super} (d)). More precisely, to get the $ \hN_{ij}^k $'s, it suffices to determine the $ \hS $-matrix.
\begin{rmk}\label{phi}
	We denote by $ \phi_n$ the positive real root of the equation $ x^2-nx-1=0 $, where $ n $ is an integer, i.e., $\phi_n=\dfrac{n+\sqrt{n^2+4}}{2} $. If an algebraic number $ \phi $ has conjugate $ -\frac{1}{\phi} $, then $ \phi $ must be of the form $ \phi_n $ for some $ n\in \mathbb{Z} $.
\end{rmk}

\begin{theorem}\label{non-selfdual}
If $ \BB $ is a rank $8$ non-self dual super-modular category, then the corresponding $ \hS $-matrix, up to relabeling the simple objects, has the following form:
$$\hS=\begin{pmatrix}1 & 1& 1 & 1\\
1& 1&-1 &-1\\
1&-1&\pm i&\mp i\\
1&-1&\mp i& \pm i\end{pmatrix}. $$

\end{theorem}

\begin{proof}
	The proof in \cite[Appendix A.2]{RSW} carries through, \emph{mutatis mutandis}.
%	\jp{I think there is a word maybe in latin to say that the proof is the same when you do the corresponding changes... I will look for it} \er{mutatis mutandis}
\end{proof}
\begin{rmk}
Having dispensed with the non-self-dual case, we assume for the rest of this section that all categories are self-dual. In particular the naive fusion coefficients are cyclically symmetric (see Corollary \ref{cyclicnaivefusion}), so we will denote $ \hN_{ij}^k$ by  $n_{i, j, k} $.
\end{rmk}

\begin{theorem}\label{norealizations}%Does not exist
%\jp{should we add hypothesis here? (self-dual) or fix them before starting with these theorems?}
%	If $G= \langle(23)\rangle$,  $\langle(01),(23)\rangle$ or $\langle(123)\rangle$, then there is no solution for the $ \hS $-matrices.\footnote{JP: should we explain better what we mean with no solutions?} 
There are no rank $8$ self-dual super-modular categories with Galois group  $G= \langle(23)\rangle$,  $\langle(01),(23)\rangle$ or $\langle(123)\rangle$.
\end{theorem}
\begin{proof}
	\begin{enumerate}
		\item[(1)] If $G= \langle(23)\rangle$,  applying
		Equation (\ref{hatSsym}) with $\sigma=\langle(2 3)\rangle$, we have the following form for the $\hS$-matrix 
		$$\hS=\begin{pmatrix}1 & d_1& d_2&d_2\\
		d_1& s_{11} & s_{12}  &\epsilon_1 s_{12}\\
		d_2&s_{12}&s_{22}&s_{23}\\
		d_2&\epsilon_1 s_{12} &s_{23} & \epsilon_2 s_{22}
		\end{pmatrix}.$$
		As 0 and 1 are fixed by $G$, by Equation (\ref{Galosactionsuper}), we know that $d_1$, $d_2$, $\dfrac{s_{11}}{d_1}$,  $\dfrac{s_{12}}{d_1}$, $\dfrac{s_{12}^2}{d_2^2}$ and $\dfrac{s_{22}s_{23}}{d_2^2}$ are rationals as they are fixed by the Galois group. Since they are also algebraic integers (see \cite[Proposition 8.13.11]{MR3242743}), we know these are integers.  Consequently, $s_{11}$, $s_{12}$, $s_{22}s_{23}$ are also integers.
		
		If $ \epsilon_1=-1 $, the orthogonality of the columns of $ \hS $ gives
		
		\[d_1(1+s_{11})=0\]
		\[ d_1d_2+s_{11}s_{12}+s_{12}s_{22}-s_{12}s_{23}=0 \]
		\[ d_1d_2-s_{11}s_{12}+s_{12}s_{23}-\epsilon_2s_{12}s_{22}=0 \]
		
		So we have $ s_{11}=-1 $. If $ \epsilon_2=1 $, then we have $ d_1d_2=0 $, which is a contradiction. If $ \epsilon_2=-1 $, we have $ d_1d_2=-s_{12}s_{22}$. Plugging this into the second equation above, we get $ s_{12}(1+s_{23})=0 $. If $ s_{12}=0 $, then $d_1d_2=0$, which is impossible. If $ s_{23}=-1 $, then $ s_{22} $ is an integer. Then all the entries of $ \hS $ are integers, which contradicts the assumption that $ G $ is $ \mathbb{Z}_2 $. \\
		If $ \epsilon_1=1 $, the orthogonality of the columns of $\hS$ gives
		\begin{center}
			$d_2^2+s_{12}^2+s_{22}s_{23}+\epsilon_2s_{22}s_{23}=0$
		\end{center}
		If $\epsilon_2=-1$, then $d_2^2+s_{12}^2=0$, a contradiction. If $\epsilon_2=1$, by applying a   Gr\"oebner basis algorithm on Maple, we get $(2s_{22}+s_{11}+1)(2d_1d_2+s_{11}s_{12}+2s_{12}s_{22}-s_{12}) =0$. One sees that if either factor is 0, we will have trivial $ G $, a contradiction.
		
		\item[(2)]  Assume $ G=\langle(01),(23)\rangle $. Using  Equation (\ref{hatSsym}), we get 
		$$\hS=\begin{pmatrix}1 & d_1& d_2&d_3\\
		d_1& \pm 1 & \pm d_2  &\pm d_3\\
		d_2&\pm d_2&s_{22}&s_{23}\\
		d_3&\pm d_3 &s_{23} & \pm s_{22}
		\end{pmatrix}.$$
		It follows from $ \hS^2=\frac{D^2}{2} I $ that $ 2d_2^2+s_{22}^2+s_{23}^2=2d_3^2+s_{22}^2+s_{23}^2 $. Since $ d_i $'s are positive, $ d_2=d_3 $.\\
		Let $$\hS=\begin{pmatrix}1 & d_1& d_2&d_2\\
		d_1& \epsilon_1 & \epsilon_2 d_2  &\epsilon_3 d_2\\
		d_2&\epsilon_2 d_2&s_{22}&s_{23}\\
		d_2&\epsilon_3 d_2 &s_{23} & \epsilon_4 s_{22}
		\end{pmatrix}.$$
This case can be eliminated using orthogonality of the columns of $\hS$. Applying a Gr\"obner basis algorithm to these equations we find that the only possible sign choice is given by 
$ \epsilon_1= \epsilon_4= 1 $ and  $ \epsilon_2= \epsilon_3=-1 $. We can further deduce that $ s_{23} = -1$, $ s_{22}=d_1 $ and $ d_1=d_2^2 $. Therefore, we have
$$\hS=\begin{pmatrix}1 & d_2^2& d_2&d_2\\
d_2^2& 1& - d_2  &- d_2\\
d_2& - d_2&d_2^2 &-1\\
d_2& - d_2&-1 &d_2^2
\end{pmatrix}$$
Notice that $ G=\Gal(\mathbb{Q}(d_2)/\mathbb{Q}) $. Computing the characteristic polynomial for $ \hN_2 $, we have
\[ p_2(x)=x^4+(-2d_2+\frac{2}{d_2})x^3+(d_2^2+\frac{1}{d_2^2}-4)x^2+(2d_2-\frac{2}{d_2})x+1  \]
Therefore, $ -2d_2+\frac{2}{d_2} $ must be an integer. In particular, $ d_2 $ satisfies a quadratic equation over $\bbQ$. This means $ \Gal(\mathbb{Q}(d_2)/\mathbb{Q}) $ is either trivial or $ \mathbb{Z}_2 $, which contradicts the fact that $ G $ is $ \mathbb{Z}_2\times\mathbb{Z}_2 $.

		\item[(3)] If  $G=\langle(123)\rangle$, then $G$ fixes $0$. Therefore $\hS_{i,0}= d_i$ are rational numbers. Since the dimensions $d_i$'s are always algebraic integers, then they must be integers in this case. Moreover,  $d_i=\hS_{0,1}=\pm\hS_{0,i+1}=\pm d_{i+1}$. So, by positivity of the dimensions (i.e. unitarity assumption), we have
		$$\hS=\begin{pmatrix}1 & d_1& d_1&d_1\\
		d_1&s_{11} & \epsilon_1s_{33} &\epsilon_2s_{22}\\
		d_1&\epsilon_1s_{33}& s_{22}&\epsilon_3 s_{11}\\
		d_1& \epsilon_2s_{22} &\epsilon_3 s_{11} & s_{33}
		\end{pmatrix}.$$
		From Corollary \ref{divisibility}, we have $d_1^2|(1+3d_1^2)$. We can deduce that $d_1=1$. Since $d_1$ is the largest (in magnitude) eigenvalue of the fusion matrices $N_1,N_2$ and $N_3$, we see that the other eigenvalues (which are real numbers) satisfy $\pm \hS_{ii}/d_1=\pm \hS_{i,i}=\pm 1$. This means the entries of $\hS$ are $\pm 1$'s which contradicts the assumption of $G$ being nontrivial.
		
	\end{enumerate}
\end{proof}

\begin{theorem}\label{trivial}%pointed
	If $G= \langle(0)\rangle$, then the corresponding $ \hS $-matrix, up to relabeling the simple objects, is one of the following:
	 $$\begin{pmatrix}1 & 1& 1&1\\
	1& -1& -1  &1\\
	1&-1&1&-1\\
	1&1&-1& -1\end{pmatrix},\quad \begin{pmatrix}1 & 1& 1&1\\
	1& 1& -1  &-1\\
	1&-1&1&-1\\
	1&-1&-1& 1\end{pmatrix}. $$
\end{theorem}
\begin{proof}
	If $ G $ is trivial, then the proof of \cite[Theorem 4.1, Case 7]{RSW} goes through \emph{mutatis mutandis} showing that the corresponding super-modular category is pointed. Thus by Proposition \ref{pointedsplits} the super-modular category splits, so that  $\hS $ has the same form as the $S$-matrix of some rank 4 pointed modular category \cite{RSW} as in the statement. 
	
%	Then, the possible $ \hat{S} $ are 
%	$$\begin{pmatrix}1 & 1& 1&1\\
%	1& -1& -1  &1\\
%	1&-1&1&-1\\
%	1&1&-1& -1\end{pmatrix} \quad     \text{and}     
%	\quad \begin{pmatrix}1 & 1& 1&1\\
%	1& 1& -1  &-1\\
%	1&-1&1&-1\\
%	1&-1&-1& 1\end{pmatrix}. $$
\end{proof}
\begin{theorem}%(01)
	If $G= \langle(01)\rangle$, then the corresponding $ \hS $ is 
	$$\begin{pmatrix}1 & 1& 2&\sqrt{6}\\
	1& 1&2 &-\sqrt{6}\\
	2&2&-2&0\\
	\sqrt{6}&-\sqrt{6}&0& 0\end{pmatrix}. $$
\end{theorem}
\begin{proof}
	By Equation (\ref{hatSsym}), we  have $$\hS=\begin{pmatrix}1 & d_1& d_2&d_3\\
	d_1&\epsilon_1 & \epsilon_2 d_2  &\epsilon_3 d_3\\
	d_2&\epsilon_2 d_2&s_{22}&s_{23}\\
	d_3&\epsilon_3 d_3 &s_{23} &  s_{33}
	\end{pmatrix}.$$
	We first assume that $\epsilon_1=1$. Then we can  have $\epsilon_2\epsilon_3=-1$ or $\epsilon_2=\epsilon_3=-1$. 
	
	For the first case, we can assume $\epsilon_2=1$, $\epsilon_3=-1$ and interchange $N_2$ and $N_3$ if necessary. Then the orthogonality of $\hS$ gives us $s_{23}(s_{22}+s_{33})=0$ and $2d_1+d_2^2-d_3^2=0$. Assume that $s_{22}+s_{33}=0$, then since the columns of $\hS$ are of equal length $2d_2^2+s_{22}^2=2d_3^2+s_{33}^2$. This gives that $d_2=d_3$, and that $d_1=0$, which is a contradiction. So we must have $s_{23}=0$. Then $\hS$ becomes \\
	$$\hS=\begin{pmatrix}1 & d_1& d_2&d_3\\
	d_1&1 & d_2  &- d_3\\
	d_2&d_2&s_{22}&0\\
	d_3&-d_3 &0 &  s_{33}\end{pmatrix}.$$
	Since $\sigma=(01)$ is the only non-trivial element of the Galois group, we conclude that \\
	$m=\dfrac{ d_2(d_1+1)}{d_1}$, $ n=\dfrac{d_3(d_1-1)}{d_1} $, $ t=\dfrac{s_{22}}{d_2} $, $ u=\dfrac{s_{33}}{d_3} $, $ v=\dfrac{d_2^2}{d_1} $, $ w=\dfrac{(d_1^2+1)}{d_1} $ and $ x=\dfrac{d_3^2 }{d_1}$ are integers as coefficients of the minimal polynomials of the $\hN_i$. Notice that $m, v, w$ and $x$ are strictly greater than $0$ and $n\ge 0$. Since $d_2+d_1d_2+d_2s_{22}=0$, we have $s_{22}<0$ so $t<0$.  Moreover, we have $t^2-u^2\neq 0$. In fact, if $t^2-u^2=0$, then $u^2+2 =\dfrac{s_{33}^2+2d_3^2}{d_3^2}=\dfrac{s_{22}^2+2d_2^2}{d_3^2}=\dfrac{d_2^2}{d_3^2}(\dfrac{s_{22}^2+2d_2^2}{d_2^2})=\dfrac{d_2^2}{d_3^2}(t^2+2)$. This implies that $d_2=d_3$. Using $2d_1+d_2^2-d_3^2=0$, we have $d_1=0$, a contradiction. Thus $t^2-u^2\neq 0$ and we have
	\begin{align*}
	    m&=-\dfrac{2t(u^2+2)}{t^2-u^2}, &n&=\dfrac{2u(t^2+2)}{t^2-u^2},
	   &v&=\dfrac{2(u^2+2)}{t^2-u^2,}\\w&=\dfrac{2(t^2u^2+t^2+u^2)}{t^2-u^2},
	    &x&=\dfrac{2(t^2+2)}{t^2-u^2}.
	\end{align*}
	Since $x>0$, we have  $t^2-u^2>0$. We have $n_{2,2,2}=\dfrac{t(t^2-u^2-2)}{(t^2-u^2)}$. In order to have $n_{2,2,2}\geq 0$, we must have $t^2-u^2-2\leq 0$. The only integer solution satisfying all the restrictions here is $t=-1$ and $u=0$. Then $s_{33}=0$ and $s_{22}=-d_2$. Thus, we have $d_1=1$. The orthogonality condition on the columns of $\hS$ gives that $2d_2-d_2^2=0$. This implies that $d_2=2$ and $d_3=\sqrt{6}$.\\
	
	If $\epsilon_2=\epsilon_3=-1$, we have\\ $$\hS=\begin{pmatrix}1 & d_1& d_2&d_3\\
	d_1&1 & -d_2  &-d_3\\
	d_2&- d_2&s_{22}&s_{23}\\
	d_3&-d_3 &s_{23} &  s_{33}
	\end{pmatrix}.$$ Similarly to the previous case, we have $m=\dfrac{d_3(d_1-1)}{d_2}$, $n=\dfrac{d_1^2+1}{d_1}$, $t=\dfrac{d_3^2}{d_1}$, $u=\dfrac{s_{22}}{d_2}$, $v=\dfrac{d_2^2}{d_1}$, $w=\dfrac{s_{33}}{d_3}$, $x=\dfrac{s_{23}}{d_2}$, $y=\dfrac{s_{23}}{d_3}$ and $z=\dfrac{d_2(d_1-1)}{d_1}$ are integers. Here we have $nv - z^2 - 2v = 0$, $t + v - 2 = 0$ and $m^2 + z^2 - 2n + 4 = 0$. Notice that $m^2+n^2\neq 0$ since $n\neq 0$. So we have $n=\dfrac{m^2+z^2}{2}+2$, $t=\dfrac{2 m^2}{m^2+z^2}$, and $v=\dfrac{2 z^2}{m^2+z^2}$. Since $t$ is an integer, we have $m^2\geq z^2$. Similarly, we have $z^2\geq m^2$. Thus $|m|=|z|$ so $t=v=1$. This means $d_2=d_3=\sqrt{d_1}$. Then $m=d_1-1$ and $d_1$ is an integer. From $|m|=|z|$, we get $d_1-1=\dfrac{d_2(d_1-1)}{d_1}$. If $d_1=1$, then we have $d_2=d_3=1$. This would force all the entries of $\hS$ to be integers, which a contradiction to the assumption that the Galois group is $\mathbb{Z}_2$. If $d_1>1$, then we have $d_2=d_1$. Recall that $d_2=d_3=\sqrt{d_1}$. This means either $d_2=d_3=d_1=0$ or $d_2=d_3=d_1=1$, again a contradiction. \\
	
	If $\epsilon_1=-1$, the orthogonality of the columns of $\hS$ gives $\epsilon_2d_2^2+\epsilon_3d_3^2=0$. Thus we have $\epsilon_2\epsilon_3=-1$ and $d_2=d_3$.  But then we have $\sigma(d_2)=\dfrac{d_2}{d_1}=-\dfrac{d_2}{d_1}$ so $d_2=0$, a contradiction.
\end{proof}

\begin{theorem}\label{Klein 4}
	 If $G= \langle(01)(23),(02)(13)\rangle$, then the corresponding $ \hS $ has the following form:
	  $$\begin{pmatrix}1 & \phi_1\phi_2& \phi_1&\phi_2\\
	 \phi_1\phi_2& 1& -\phi_2  &-\phi_1\\
	 \phi_1&-\phi_2&-1&\phi_1\phi_2\\
	 \phi_2&-\phi_1&\phi_1\phi_2& -1\end{pmatrix}.$$	 
\end{theorem}
\begin{proof}
	By Equation (\ref{hatSsym}), we have the corresponding $ \hS $: $$\begin{pmatrix}1 & d_1& d_2&d_3\\
	d_1& \epsilon_1& \epsilon_2 d_3  &\epsilon_3 d_2\\
	d_2&\epsilon_2 d_3&\epsilon_4&\epsilon_5 d_1\\
	d_3&\epsilon_3 d_2&\epsilon_5 d_1& \epsilon_6\end{pmatrix}. $$
	Using orthogonality of the columns of $ \hS $ and the fact that $ d_i\geq 1 $, there are only 2 possibilities for $ \epsilon_i $'s, namely,
	\begin{enumerate}
	\item $ \epsilon_1=1,\epsilon_2=-1,\epsilon_3=-1,\epsilon_4=1,\epsilon_5=-1,\epsilon_6=1, $ or
		\item $ \epsilon_1=1, \epsilon_2=-1,\epsilon_3=-1,\epsilon_4=-1, \epsilon_5=1,\epsilon_6=-1 .$\\
	\end{enumerate}
		For the first case,  the orthogonality of $ \hS $ gives $d_1=d_2d_3$, $d_2=d_1d_3$ and $d_3=d_1d_2$. So we have $d_1d_2d_3=(d_1d_2d_3)^2$, we have $d_1d_2d_3=1$. Since $ d_i\geq1 $ for all $ i $,  this implies that $d_1=d_2=d_3=1$. This cannot happen since the corresponding Galois group should be trivial, which is a contradiction to our assumption.\\
		
	Consider the second case. The orthogonality of $\hS$ gives $d_1=d_2d_3$. So we can write the corresponding matrix as $$\hS=\begin{pmatrix}1 & d_2d_3& d_2&d_3\\
	d_2d_3& 1& -d_3  &-d_2\\
	d_2&-d_3&-1&d_2d_3\\
	d_3&-d_2&d_2d_3& -1\end{pmatrix}. $$
	Notice that Equation (\ref{Galosactionsuper}) indicates that $ d_2 $ and $ -1/d_2 $ are conjugates. By Remark \ref{phi}, we know that $ d_2 =\phi_m$ for some $ m \in\mathbb{Z}$. Similarly, $ d_3=\phi_n $ for some integer $ n $.
	
	Thus we have
	$$\hS=\begin{pmatrix}1 & \phi_m\phi_n& \phi_m&\phi_n\\
	\phi_m\phi_n& 1& -\phi_n  &-\phi_m\\
	\phi_m&-\phi_n&-1&\phi_m\phi_n\\
	\phi_n&-\phi_m&\phi_m\phi_n& -1\end{pmatrix}. $$

The corresponding $\hN_{i}$ matrices have integer entries in terms of $m$ and $n$. More precisely, we have \\

	$\hat{N_1}=\begin{pmatrix}0 & 1& 0&0\\
	1& mn& m  &n\\
	0&m&0&1\\
	0&n&1&0\end{pmatrix}, $
	$\hat{N_2}=\begin{pmatrix}0 & 0& 1&0\\
	0& m& 0  &1\\
	1&0&m&0\\
	0&1&0&0\end{pmatrix},$
and 	$\hat{N_3}=\begin{pmatrix}0 & 0& 0&1\\
	0& n& 1  &0\\
	0&1&0&0\\
	1&0&0&n\end{pmatrix}.$\\
	
	Using the  formula given in Lemma \ref{FSS} , we calculate the 2nd Frobenius-Schur indicator for the simple object $ X_2 $:
	\[\nu_2(X_2)=\pm1=\dfrac{2}{D^2} \left(d_2\left(\dfrac{1}{\theta_2}\right)^2+md_1^2+d_1d_3\left(\dfrac{\theta_1}{\theta_3}\right)^2+md_2^2+d_1\theta_2^2+d_1d_3\left(\dfrac{\theta_3}{\theta_1}\right)^2\right)\] from this we obtain
	\begin{align*}
 \pm\dfrac{D^2}{2} &=m\left(d_1^2+d_2^2\right)+d_2\left(\theta_2^2+\theta_2^{-2}\right)+d_1d_3\left(\left(\dfrac{\theta_1}{\theta_3}\right)^2+\left(\dfrac{\theta_1}{\theta_3}\right)^{-2}\right) \\
	&=m\left(d_2^2d_3^2+d_2^2\right)+2d_2\operatorname{Re}\left(\theta_2^2\right)+2d_2d_3^2\operatorname{Re}\left(\dfrac{\theta_1}{\theta_3}\right)^2\\&\le \dfrac{D^2}{2}=1+d_2^2d_3^2+d_2^2+d_3^2\\
   \Rightarrow 0&\ge md_2^2\left(d_3^2+1\right)+2d_2\operatorname{Re}\left(\theta_2^2\right)+2d_2d_3^2\operatorname{Re}\left(\dfrac{\theta_1}{\theta_3}\right)^2-1-d_2^2d_3^2-d_2^2-d_3^2\\
	&=md_2^2\left(d_3^2+1\right)-2d_2\left(d_3^2+1\right)-d_2^2\left(d_3^2+1\right)-\left(d_3^2+1\right)\\&=\left(md_2^2-2d_2-d_2^2-1\right)\left(d_3^2+1\right)\\
	\\
	\Rightarrow 0&\ge \left(md_2^2-2d_2-d_2^2-1\right)\\
	&=d_2^2\left(m-1\right)-2d_2-1\\
	&=\phi_m^2\left(m-1\right)-2\phi_m-1\\
	&=\left(m\phi_m+1\right)\left(m-1\right)-2\phi_m-1\\
	&=\left(m-2\right)\left(\phi_m\left(m+1\right)+1\right).
	\end{align*}        
	Thus $ m $ must be $ 0, 1,$ or $ 2$.\\
	
	Similarly, we calculate the 2nd Frobenius-Schur indicator for $ X_3$: \\
	$ \nu_2(X_3)=\pm1=\dfrac{2}{D^2} \left(d_3\theta_3^{-2}+nd_1^2+d_1d_2\left(\dfrac{\theta_1}{\theta_2}\right)^2+d_1d_2(\dfrac{\theta_2}{\theta_1})^2+nd_3^2+d_3\theta_3^2
	\right)$
	\begin{align*}
	\Rightarrow \pm\dfrac{D^2}{2} &= 2d_3\operatorname{Re}(\theta_3^2)+n(d_2^2d_3^2+d_3^2)+2d_1^2d_3\operatorname{Re}\left(\dfrac{\theta_1}{\theta_2}\right)^2\\&\le \dfrac{D^2}{2}=1+d_2^2d_3^2+d_2^2+d_3^2\\
 \Rightarrow 0\ge& 2d_3\operatorname{Re}(\theta_2^2)+nd_3^2(d_2^2+1)+2d_3d_2^2\operatorname{Re}\left(\dfrac{\theta_1}{\theta_2}\right)^2-1-d_2^2d_3^2-d_2^2-d_3^2\\
	\ge&-2d_3+nd_3^2(d_2^2+1)-2d_3d_2^2-d_3^2(d_2^2+1)-(1+d_2^2)\\
	&=(nd_3^2-2d_3-d_3^2-1)(d_2^2+1)\\
	\Rightarrow 0\ge&(nd_3^2-2d_3-d_3^2-1) \\
	&=d_3^2(n-1)-2d_3-1\\
	&=\phi_n^2(n-1)-2\phi_n-1\\
	&=(n\phi_n+1)(n-1)-2\phi_n-1\\
	&=(n-2)(\phi_n(n+1)+1)
	\end{align*}        
	So $ n $ must  be $ 0, 1,$ or $ 2$.\\    
	Up to symmetry, we can exclude the cases $(m,n) =(0,0) ,(1,1), (1,0), (2,2)$ since the corresponding Galois groups are not isomorphic to $ \mathbb{Z}_2\times \mathbb{Z}_2 $. The possible value for this case, up to symmetry, is $ (m,n)=(1,2) $. Notice that $\phi_1= \dfrac{1+\sqrt{5}}{2}$ and $\phi_2=1+\sqrt{2}$.

\end{proof}

In the last few cases we were unable to complete the classification in general--instead we placed bounds on the $\hN_{ij}^k$'s. Since $N_{ij}^k\leq 2||N_i||_{\max}$, this could also be done in terms of bounds on the $N_i$'s. Sometimes it is easier to work in terms of a bound on the dimensions $d_i$.  Indeed, the proof of \cite[Lemma 3.14]{BNRW1} goes through with no change, from which we conclude:
$\hN_{ij}^k\leq d_i\leq 4||\hN_i||_{\max}$.

\begin{theorem}% Z4
 If $G=\langle(0123)\rangle$ and  $\hN_{ij}^k< 14$, the corresponding $ \hS $  is  $$ \begin{pmatrix}1 & d_1& d_2&d_3\\
 	 d_1&- d_2 &  d_3  &1\\
 	 d_2& d_3&-1&- d_1\\
 	 d_3&1 &- d_1 &  d_2
 	 \end{pmatrix},$$
 	 where $d_1=1+\sqrt{2}+\sqrt{2+\sqrt{2}}$, $d_2=1+\sqrt{2}+\sqrt{2(2+\sqrt{2})}$, and $d_3=1+\sqrt{2+\sqrt{2}}$.
\end{theorem}

\begin{proof}
	 Applying
	Equation (\ref{hatSsym}) with $\sigma=\langle(0 1 2 3)\rangle$, we have the following form of $\hS$ matrix    $$\hS=\begin{pmatrix}1 & d_1& d_2&d_3\\
	d_1& \epsilon_1 d_2 & \epsilon_2 d_3  &\epsilon_3\\
	d_2&\epsilon_2 d_3&\epsilon_4&\epsilon_5 d_1\\
	d_3&\epsilon_3 &\epsilon_5 d_1 & \epsilon_6 d_2
	\end{pmatrix}.$$
	
	Using a Maple's Gr\"obner basis algorithm, we deduce that $ \epsilon_1= \epsilon_4= \epsilon_5=-1 $ and  $ \epsilon_2= \epsilon_3= \epsilon_6=1 $.

	So  $$ \hS =\begin{pmatrix}1 & d_1& d_2&d_3\\
	d_1&- d_2 &  d_3  &1\\
	d_2& d_3&-1&- d_1\\
	d_3&1 &- d_1 &  d_2
	\end{pmatrix}.$$
	Let $ p_1(x)= x^4-c_1x^3+c_2x^2+c_3x-1 $ be the characteristic polynomial of $ \hat{N_1} $. Then $ p_3(x)=x^4-c_3x^3-c_2x^2+c_1x-1 $, where $ c_i\in \mathbb{Z} $ for $ i=1,2 $ and $ 3 $.  Notice that $ c_1=\text{Trace}(\hat{N_1}) \ge 0$ and $ c_3=\text{Trace}(\hat{N_3}) \ge 0$ as the $ \hN_i $'s are matrices with nonnegative integer entries.  Let $ p_2(x)=x^4-b_1x^3+b_2x^2+b_3x+1 $  be the characteristic polynomial of $ \hat{N_2} $, where\\ $ b_1=b_3=d_2+\frac{d_3}{d_1}-\frac{1}{d_2}-\frac{d_1}{d_3} $ and $ b_2=-2+\frac{d_1}{d_2d_3}-\frac{d_3}{d_1d_2}-\frac{d_2d_1}{d_3}+\frac{d_2d_3}{d_1} $. 
	
	The orthogonality of the rows of $ \hat{S} $ gives $ d_1=d_1d_2-d_2d_3-d_3 $, $ d_3=-d_1+d_1d_2-d_2d_3 $, $ d_1d_2=d_3+d_1+d_2d_3 $ and $ d_2d_3=-d_1+d_1d_2-d_3 $. So we have $ b_2= -6 $ and $ b_3=-b_1 $. Thus $ p_2(x)=x^4-b_1x^3-6x^2+b_1x+1 $, where $ b_1=\text{Trace}(\hat{N_2}) \ge 0$. 
	
	Notice that $c_1+c_3=2\dfrac{(d_2+1)d_3}{d_2} +4\dfrac{d_2}{(d_2+1)d_3}$. This gives $ c_1+c_3 \ge 4\sqrt{2}$. Since $ c_1 $ and $ c_3 $ are integers, we have $ c_1+c_3\ge 6 $. Moreover, we have $ 4b_1-c_1^2+8c_2+c_3^2=0 $.     
	
	Let $ \Delta =c_1-c_3$ and $ \Sigma =c_1+c_3 $, then $ c_2=\frac{1}{16}[3\Delta\Sigma\pm \sqrt{(32+\Delta^2)(-32+\Sigma^2)}] $ and $ b_1=\frac{1}{8}[-\Delta\Sigma\mp\sqrt{(32+\Delta^2)(-32+\Sigma^2)}] $. 
	Let $ P=\frac{16c_2-3\Delta\Sigma}{\Delta^2+32}=\pm \sqrt{\frac{\Sigma^2-32}{\Delta^2+32}} $.  
	
We compute the $ n_{i, j, k} $'s and we get the following relations:
\begin{align*}
    	 n_{1,1,1}&=\dfrac{5 c_1-3 c_3}{8} -\dfrac{ (c_1-c_3)P}{8} \\n_{1,1,2}&=1-P =1+n_{1,2,3}=2+n_{2,3,3}\\
		 n_{1,1,3}&=\dfrac{c_1+c_3}{8}-\dfrac{(c_1-c_3)P }{8} =n_{1,3,3}=\frac{1}{2}(n_{1,1,1}+n_{3,3,3})\\
		 n_{1,2,2}&=\dfrac{c_1+c_3}{4}+\dfrac{(c_1-c_3)P}{4} =n_{2,2,3}\\
	 n_{2,2,2}&= \dfrac{c_1^2-c_3^2}{4} -2c_2+2 P=b_1+2P
\end{align*}

%*	$ n_{1,2,3}= -P$\\
	
%*	$ n_{1,3,3}=\dfrac{c_1+c_3}{8}-\dfrac{(c_1-c_3)P}{8}  $\\	
%*	$ n_{2,2,3}=\dfrac{c_1+c_3}{4}+\dfrac{(c_1-c_3)P}{4} $\\
	
%*	$ n_{2,3,3}=-(1+P) $\\
	
%*	$ n_{3,3,3}=\dfrac{ 5 c_3-3 c_1}{8}-\dfrac{(c_1-c_3)P }{8} $\\
	
	Recall that the fusion coefficients are integral. In particular, since $ n_{2,2,2} $ is an integer, we know that $ c_1 $ and $ c_3 $ are both even. Thus $\Delta$ and $\Sigma$ are divisible by 2. Via a computer search for integer solutions using the above equations, we found there is only one solution when $ n_{i, j, k}<  14$, with $c_1=c_3=4$ and $c_2=2P=-2$. The corresponding $ \hS $ matrix for this case is the one in the statement (and is the same as that of $\PSU(2)_{14}$).\end{proof}

We can make further progress using more sophisticated number theoretical arguments:
\begin{lem}\label{div4}
If $\Sigma$ and $\Delta$ are divisible by 4, the corresponding super-modular categories have $c_1=c_3=\sqrt{2}(\zeta^{2i-1}-\overline{\zeta}^{2i-1})$, $c_2=-(\zeta^{2i-1}+\overline{\zeta}^{2i-1})$ and $P=-\frac{1}{2}(\zeta^{2i-1}+\overline{\zeta}^{2i-1})$, where $\zeta=1+\sqrt{2}$, $\overline{\zeta}=1-\sqrt{2}$ and $i\geq 1$ is an integer.
\end{lem}

\begin{proof}
 Assume that $ \Sigma $ and $ \Delta $ in the proof above are also divisible by 4.  Denote $a=\frac{\Sigma}{4}$, $b=\frac{\Delta}{4}$ and $c=P$. Then we have the following Diophantine equation  \[a^2-(b^2+2)c^2=2.\]
 Lemma \ref{beq0} below shows that $ b=0$. Consequently, we have $c_1=c_3,$ and the Diophantine equation becomes $a^2-2c^2=2$. Since $a=\frac{c_1}{2}\geq 0$ and $c= P = \frac{c_2}{2}\leq -1$ the resulting solutions are $$a(i):=\frac{1}{\sqrt{2}}(\zeta^{2i-1}-\overline{\zeta}^{2i-1}),\quad c(i)=-\frac{1}{2}(\zeta^{2i-1}+\overline{\zeta}^{2i-1}),$$ where $1\leq i$ and $\zeta=1+\sqrt{2}$ and $\overline{\zeta}=1-\sqrt{2}$.  This determines all possible fusion rules under these assumptions.  The first few are $(a,c)\in\{(2,-1),(10,-7),(58,-41),(338,-239),\ldots\}.$

\end{proof}
Some cases can be ruled out if we assume the MME conjecture using Lemma \ref{extension lemma} as follows.
\begin{example}
In the case $(a,c)=(58,-41)$, we find that $d_1$ is a root of the irreducible polynomial $x^4-2\cdot58x^3-82x^2+2\cdot58x-1$.  The smallest cyclotomic field in which $d_1$ resides has degree $464=2^4\cdot 29$ (i.e., the conductor of $\bbQ(d_1)$ is $464$).  Now suppose that the corresponding super-modular category $\BB$ has a MME $(\CC,f)$.  Then the order of the $T$ matrix of $\CC$ is divisible by $29$, so that $7\mid\varphi(29)\mid[\bbQ(T):\bbQ]$.  But Lemma \ref{extension lemma} and the results of \cite{ng2010congruence} imply that $[\bbQ(T):\bbQ]=2^m$ for some $m$ (since $[\bbQ(T):\bbQ(S)]=2^t$).  Thus no such category can exist.
\end{example}
\begin{rmk}
The $(a,c)=(10,-7)$ case cannot be dealt with in this way since the corresponding conductor is $80$.
\end{rmk}

\begin{lem}\label{beq0}
	Assume $ a $, $ b $ and $ c $ are integers and  $ a^2-(b^2+2)c^2=2 $, then $ b=0 $.
\end{lem}

\begin{proof}
		Reducing modulo 8 both sides of the equation, there are three cases to consider since a square modulo 8 is $ 0 $, $ 1 $, or $ 4 $.
	\begin{itemize}
		\item If $ b^2\equiv 1\mod 8 $, then we have $ a^2-2\equiv 3c^2\mod 8 $. This gives no solutions.
		\item If $ b^2\equiv 0\mod 8 $, then  we have $ c\equiv 1\mod 8 $ and $ a\equiv 4 \mod 8 $.
		\item If $ b^2\equiv 4\mod 8 $, then  we have $ c\equiv 1\mod 8 $ and $ a\equiv 0 \mod 8 $.
	\end{itemize}
	
	Therefore, we must have that $ a $ and $ b $ are even and $ c $ is odd. Moreover, if $ 4|b$, then $ 4\nmid a $ and vice versa.
	
	Now we consider both sides of $ a^2-(b^2+2)c^2=2 $ modulo 4. This gives us $ b^2+2\equiv 2 \mod 4 $. Let $ B=b^2+2 $, and then we need to solve the following Pell-like equation
	\[a^2-Bc^2=2\]
	As $ b $ is even, $ B $ is not divisible by 4. So we write $ B=m^2d $, where $ d $ is square-free and even and m is odd. 
	
	Claim: $ d=2 $. Assume otherwise, then we can prove that $ a^2-Bc^2=2 $ has no solutions by looking at the class group of $ \mathbb{Z}[\sqrt{d}] $ via genus theory. In fact, assume $ d\neq 2 $ and even. Then the equation $ a^2-d(mc)^2 =2 $ can be written as 
	\[a^2-dy^2=2.\]
	If the above equation has no integer solution, then $ a^2-Bc^2=2 $ has no solution. Now we consider the quadratic number field $ K=\mathbb{Q}(\sqrt{d})$. We denote the class group  of $K $ by $ C_K $ (see \cite{frohlich1993algebraic} Page 45), which is a finite abelian group. Let $ V=(\mathbb{Z}/2\mathbb{Z})^g $, where $ g $ is the number of distinct prime dividing $ d$. Let $ e_i =(0,\ldots,1,\ldots,0)$ be the basis of $ V $, where $ i=1,\ldots,g $ and 1 is on the $n ^{th}$ position. Let $ C_{K,2} $ be the subgroup of $ C_K $ consisting of the elements of order 2. For primes $ p_1,\ldots, p_g \in \mathbb{Z} $, denote the corresponding prime ideals as $ \mathfrak{p}_1, \ldots, \mathfrak{p}_g \in\mathbb{Z}[\sqrt{d}] $.  Define the map 
	\begin{align*}
	\phi: V &\rightarrow C_{K,2}\\
	e_i&\mapsto [\mathfrak{p}_i].
	\end{align*}
	
	This assignment gives a group homomorphism. 
	By Corollary 1 in Chapter 5 of \cite{frohlich1993algebraic}, we know that $ \phi $ is surjective and $ \ker(\phi)=\{0,(1,1,\ldots,1)\} $. Consequently, $ C_{K,2}\simeq (\mathbb{Z}/\mathbb{Z}_2)^{g-1}  $. In particular, if  $ g\geq 2 $, then for any prime $ p|d$, $ \mathfrak{p}=(p,\sqrt{d}) $ is not principal.
	
	Now we return to our equation $a^2-dy^2=2 $, where $ d\neq2 $ and even. Consider the ideal $ (a+y\sqrt{d})\subseteq\mathbb{Z}[\sqrt{d}] $, which has norm 2. We have $ (a+y\sqrt{d})(a-y\sqrt{d}) =(2)$. Moreover, we have $ (2,\sqrt{d})^2=(2) $. By the unique factorization, we have $ (2,\sqrt{d})=(a+y\sqrt{d}) $. However, if $ g\geq2 $, $ (2,\sqrt{d}) $ is not principal. Consequently, there is no integer solutions for $ a $ and $ y $ when $ d\neq 2 $.

	Thus we have
	\begin{align*}
	    a^2-2m^2c^2&=2   &b^2-2m^2&=-2.
	\end{align*}

	One can further deduce that $ 4|b $. Let $ b=4\beta $, the second equation gives us $ m^2-8\beta^2=1 $. This is a Pell-equation. Notice that $ (m,\beta)=(3,1) $ is the smallest non-trivial solution. Let $ z=3+2\sqrt{2} $ and denote its conjugate as $ \bar{z} $. The solutions $ (m,\beta) $ of the equation are given by 
	
	\begin{align*}
	    m_n&= \dfrac{z^n+\bar{z}^n}{2} &\beta_n&=\dfrac{z^n-\bar{z}^n}{4\sqrt{2}},
	\end{align*}
	
	where $ n $ is a positive integer.
	We also have  $ a^2-2y^2=2 $,  which is a Pell-type equation. Notice that $ (a,y)=(2,1) $ is a solution. Let $ s=2+\sqrt{2} $. By the theorem of  K. Mahler \cite{mahler1935grossten}, the solutions are given by 
	
	\begin{align*}
	    a_k&=\dfrac{s^k+\bar{s}^{k}}{2\sqrt{2^{(k-1)}}} &y_k&=\dfrac{s^k-\bar{s}^{k}}{2\sqrt{2^k}},
	\end{align*}
	
	where $ k $ is an odd positive integer. By modifying the indices, we know the solutions of the pair $ (m_n,y_n) $ are given by 
	
	\begin{align*}
	    y_n&=\dfrac{(z+1)^{2n+1}-(7-z)^{2n+1}}{2^{3n+2}\sqrt{2}} &m_n&=\dfrac{z^n+(6-z)^n}{2},
	\end{align*}
	
	where $ n\in \mathbb{N} $.
	Recall that the values of $ m $ and $ y $ are related by $ y=mc $, where $ m $ and $ c $ are both odd. In particular, $ y\geq m $. Now we consider the function given by $ f(x)=\dfrac{y_x}{m_x} $. Using standard calculus, we know that $ f $ is a monotonic increasing function and $ \displaystyle{\lim_{x\to \infty} f(x)=1+\sqrt{2}} $. Therefore, the only possible solution here is $ m=1$. Consequently, we have $ b=0 $.
\end{proof}

\begin{rmk}
If $n_{i, j, k}<115$, by a computer search for positive integer values, we find two more solutions with $(\Sigma,\Delta)=(40,0)$ and $(232,0)$, which are correspond to $i=2,3$ in Lemma \ref{div4}.  The first possible solution with $\Sigma\equiv 2\pmod{4}$ has $(\Sigma,\Delta)=(434,18)$ and $n_{1,1,1}=115$.
\end{rmk}

\begin{theorem}%Z3
	 If $G=\langle(012)\rangle$ and  $\hN_{ij}^k < 21$, then $ \hS $  is  $$\begin{pmatrix}1 & d& 1+d &d^2-1\\
	 d&  -(1+d) &-1 &d^2-1\\
	 1+d&-1 & d&-(d^2-1)\\
	 d^2-1& d^2-1 &- (d^2-1) & 0
	 \end{pmatrix},$$
	 where $d$ is the largest real root of the polynomial $x^3 - 3x - 1=0$.
	 
\end{theorem}

\begin{proof}
	Applying Equation (\ref{hatSsym}) to $\sigma = (012)$, we get $$\hS=\begin{pmatrix}1 & d_1& d_2&d_3\\
	d_1& \epsilon_1 d_2 &\epsilon_2  &\epsilon_3 d_3\\
	d_2&\epsilon_2&\epsilon_4 d_1&\epsilon_5 d_3\\
	d_3&\epsilon_3 d_3 &\epsilon_5 d_3 & s_{33}
	\end{pmatrix}.$$
A computation using $\hS^2=\frac{D^2}{2} I$ and  $d_i\geq 1$ reduces the sign choices to the following 3 cases:
\begin{enumerate}
\item[(1)]$ \epsilon_3=\epsilon_4=-1, \epsilon_1=\epsilon_5=1, \epsilon_2=-1 $,

\item[(2)] $ \epsilon_3=\epsilon_4=1, \epsilon_1=\epsilon_5=-1, \epsilon_2=-1 $, or
\item[(3)] $ \epsilon_3=\epsilon_4=-1, \epsilon_1=\epsilon_5=-1, \epsilon_2=1 $.
\end{enumerate}
In case (3), we find that $d_3^2+d_1d_2-(d_1+d_2)=0$.  However, since $d_i\geq 1$, we have $d_3^2+d_1d_2 \geq 2$ and $-(d_1+d_2)\leq -2$. So, the equality holds if and only if $d_1+d_2 = 2 = d_3^2+d_1d_2$, which forces $d_1=d_2 = d_3=1$.
%we have $d_1d_2\geq d_1$ so that $0\geq d_3^2-1$, hence $d_3=1$. However, from this we can infer $d_1=d_2=1$ since $d_3^3/(d_1d_2)\in\Z$.
This is impossible since the Galois group is non-trivial by hypothesis.

Case (1) is equivalent to case (2) by permuting columns/rows $2$ and $3$ and relabeling $d_1\leftrightarrow d_2$.  
So, without loss of generality, we may assume we are in case (2). Let $g(x)=x^3-c_1x^2+c_2x-c_3  $ be an irreducible polynomial for which $ d_3 $
is a root. Note that $c_1=\dfrac{d_3}{d_1d_2}(d_1d_2+d_2-d_1)  $, $ c_2=\dfrac{d_3^2}{d_1d_2}(d_2-d_1-1)$, and $ c_3=-\dfrac{d_3^3}{d_1d_2} $. The orthogonality of the rows of $ \hS $ shows that $ c_1=-c_3 $. Moreover, $ \dfrac{c_2}{c_3}=-\lambda_{33}\in \mathbb{Z} $. Let $ n= \lambda_{33}$ and $c=-c_3=c_1$, so we have $ g(x) =x^3-cx^2+ncx+c$.  Since the Galois group is $ \mathbb{Z}_3 $, we have that $ \dfrac{dis(g)}{c^2}=c^2(n^2+4)-2nc(9+2n^2)-27$ is a square.\\
Take $ t $ to be the positive root of this, that is, $ t=\dfrac{(d_1-1)(d_1+d_2)(1+d_2)}{d_1d_2} $.\\

Notice that  $ c=\dfrac{d_3^3}{d_1d_2} >0 $. Moreover $t>0 $. Computing the fusion rules, we get
\begin{align*}
     n_{1,1,1} &=\dfrac{(t-nc-1)}{2}-\dfrac{t}{n^2+3} & n_{1,1,2}&=n_{1,3,3} =\dfrac{-c n+2 n^2+t-3}{2 \left(n^2+3\right)} \\
 n_{1,1,3}&= \dfrac{c n^2+2 c-n t+3 n}{2 \left(n^2+3\right)} &
n_{1,2,2}&=n_{2,3,3}= \dfrac{c n-2 n^2+t+3}{2 \left(n^2+3\right)}\\
n_{1,2,3}&= \dfrac{c-3 n}{n^2+3}&
n_{2,2,2}&=\dfrac{1+nc+t}{2}-\dfrac{t}{3+n^2}\\
n_{2,2,3}&=\dfrac{2c+3n+cn^2+nt}{2(3+n^2)} 
&n_{3,3,3}&=\dfrac{c+n^3}{n^2+3}
\end{align*}

If we restrict $n_{i, j, k}<21$, the only integer values of $n,t$ and $c$  that satisfy $t^2=c^2(n^2+4)-2nc(9+2n^2)-27$ and yield $n_{i, j, k}\in\Z$ is $(n,t,c)=(0,3,3)$. The corresponding $\hS$-matrix is the one given in the statement and is the same as that of $\PSU(2)_7$ (see \cite{RSW}).
\end{proof}

\begin{rmk}
Here is an alternative approach that is less computationally intensive, but assumes the minimal modular extension conjecture holds.  First notice that $c$ is a divisor of $\dim(\CC)$, so that if we assume the MME conjecture holds then, by the Cauchy theorem \cite{BNRW1}, any prime divisor $p$ of $c$ must divide the order $N$ of the $T$-matrix of any minimal modular extension of the corresponding super-modular category.  Now, by Lemma \ref{extension lemma}, we have $\varphi(N)=[\mathbb{Q}(T):\mathbb{Q}]=3\cdot 2^k$ since $|G|=3$.  Thus if $p\mid c$, we also have $\varphi(p)=2^a3^b$ where $b\in\{0,1\}$ and at most one prime divisor $p$ can have $3\mid\varphi(p)$.  Thus the prime divisors of $c$ are somewhat uncommon (for example Fermat primes).

For $n=0$, the discriminant equation above yields the Diophantine equation $(2c)^2-27=t^2$, which has finitely many solutions.  The only values of $c>0$ that correspond to a solution are: $3$ and $7$.
Since $n_{3,3,3}\in\Z$, when $n=0$ we have $3\mid c$. So $c=3$ which, in turn, implies $t=3$, giving the same solution as above.  So in this case we do not need to assume the MME conjecture.

For $n=1$ the Diophantine discriminant equation $5c^2-22c-27=t^2$ has infinitely many solutions, with the smallest few $c$ values: 
$$c\in\{7,31,199,1351,9247,63367,434311,2976799,20403271\}.$$
 The method above eliminates all of these values of $c$ except for $7$ (notice that $9\mid\varphi(1351)=2^73^2$).
 In the case that $c=7$, we find that $t=8$ which implies $n_{1,1,1}=-2$, so this cannot occur.
\end{rmk}

\begin{theorem}%Z2
If $G= \langle(01)(23)\rangle$ and $ d_i < 14 $ for all $ i $, then the corresponding $ \hS $ is one of the following:
$$\begin{pmatrix}1 & \phi_1^2& \phi_1&\phi_1\\
\phi_1^2& 1& -\phi_1  &-\phi_1\\
\phi_1&-\phi_1&-1&\phi_1^2\\
\phi_1&-\phi_1&\phi_1^2& -1\end{pmatrix},\quad\begin{pmatrix}1 &  \phi_2^2&  \phi_2& \phi_2\\
\phi_2^2& 1& - \phi_2  &- \phi_2\\
\phi_2&- \phi_2&-1& \phi_2^2\\
\phi_2&- \phi_2& \phi_2^2& -1\end{pmatrix},$$ 

$$\begin{pmatrix}1 &   \phi_1& 1 & \phi_1\\
\phi_1& -1& \phi_1  &-1 \\
1& \phi_1&-1&  -\phi_1\\
\phi_1&- 1&  -\phi_1& 1\end{pmatrix}, \quad\begin{pmatrix}1 &    \phi_2& 1 &  \phi_2\\
\phi_2& -1&   \phi_2  &-1 \\
1& \phi_2&-1&   -\phi_2\\
\phi_2&- 1&   -\phi_2& 1\end{pmatrix}. $$

\end{theorem}

\begin{proof}
Similar as previous cases, we have $$\hS=\begin{pmatrix}1 & d_1& d_2&d_3\\
	d_1&\epsilon_1 & \epsilon_2 d_3  &\epsilon_3 d_2\\
	d_2&\epsilon_2 d_3&s_{22}&s_{23}\\
	d_3&\epsilon_3 d_2 &s_{23} &  s_{33}
	\end{pmatrix}.$$

		\textbf{Case} (1) $ \epsilon_1=1 $. Using Maple's Gr\"obner basis algorithm, we deduced that  \[ (s_{33}+1)(s_{23}-1)(s_{23}+1)=0.  \] First, we assume $ s_{33}+1=0 $, then we have $ s_{33}=s_{22}=-1 $, $ \epsilon_2=\epsilon_3=-1 $, $ \epsilon_1=1 $ and $ s_{23}=d_1=d_2d_3 $. Therefore the corresponding $\hS$ is given by
		$$\hS=\begin{pmatrix}1 & d_2d_3& d_2&d_3\\
		d_2d_3& 1& -d_3  &-d_2\\
		d_2&-d_3&-1&d_2d_3\\
		d_3&-d_2&d_2d_3& -1\end{pmatrix}. $$
		Notice that this is exactly the same matrix we derived in Theorem \ref{Klein 4}. But here we do not get a contradiction since the Galois group \emph{is} $\Z_2$.  Thus the same argument using the 2nd Frobenius-Schur indicator works here. Since the Galois group is $ \mathbb{Z}_2 $, we have solutions for $ S $-matrix when $(m,n) =(1,1), (1,0),(2,0)$ and $ (2,2)$, i.e. $(d_2,d_3)=(\phi_i,\phi_i)$ or $(\phi_i,1)$ for $i=1,2$.  The cases $(1,1)$ and $(2,2)$ yield the first two $\hS$-matrices above, while for $(2,0)$ and $(1,0)$ the Galois group $G\neq \langle (01)(23)\rangle$, a contradiction.  However, see Case 2 below where these solutions do occur. 
		
		If $ s_{23}-1=0 $, one can show that the corresponding Galois group is trivial.

		Now we assume $ s_{23}+1=0 $, then the matrix $ \hS $ has the form 
		$$\hS=\begin{pmatrix}1 & d_3^2& d_3&d_3\\
		d_3^2& 1& -d_3  &-d_3\\
		d_3&-d_3&d_3^2&-1\\
		d_3&-d_3&-1& d_3^2\end{pmatrix}. $$
		Notice this is the same matrix as the previous one if $ d_2=d_3 $ and permuting the matrices $ \hN_2 $ and $ \hN_3 $.
		
		\textbf{Case} (2) $ \epsilon_1=-1 $. In this case, the $ \hS $ is of the form $$\hS=\begin{pmatrix}1 & d_1& d_2&d_3\\
		d_1&-1& d_3 &- d_2\\
		d_2& d_3&s_{22}&s_{23}\\
		d_3&- d_2 &s_{23} & - s_{22}
		\end{pmatrix}.$$
		Notice that the conjugate of $d_1  $ is $ -\dfrac{1}{d_1} $. Moreover, we know that if $ d_1=1 $, then the corresponding Galois group is trivial. Thus the field $ \mathbb{Q}(\hS)=\mathbb{Q}(d_1) $, where $ d_1=\phi_n=\dfrac{n+\sqrt{n^2+4}}{2} $ for some $ n $. Now we assume $ k\sqrt{P}=\sqrt{n^2+4} $, where $k$ is an integer and $ P $ is a square-free integer. Then  $ d_1=\dfrac{n+k\xi}{2} $, where $ \xi=\sqrt{P} $. Then $ \mathbb{Q}(\hS)=\mathbb{Q}(\xi) $. As all the entries of $ \hS $ are algebraic integers, we can assume $ d_2=a+b\xi$, $ d_3=c+d\xi $, $ s_{22}=e+f\xi$, $ s_{23}=g+h\xi $, where $ a, b, c, d, e, f, g $ and $ h $ are either half integers or integers. Then using Maple's Gr\"obner basis algorithm to eliminate non-rational variables we obtain 21 Diophantine equations (over $\frac{1}{2}\Z$).
		
		\iffalse we have the following relations: 
		\[gk-hn+2f=0\]
		\[b^2P+d^2P-a^2-c^2=0\]
		\[2b^2h-b^2k+2d^2h+d^2k=0\]
		\[-a d n+b c n+2 a b+2 c d=0\]
		\[d^2 n P-2 b d P-c^2 n-2 a c=0\]
		\[b d n P+2 d^2 P-a c n-2 a^2=0\]
		\[ a k-b n-2 d =0\]
		\[c k-d n+2 b=0 \]
		\[d k P-c n-2 a=0 \]
		\[b k P-a n+2 c=0\]
		\[2 a b h-b^2 n+2 c d h+d^2 n=0\]
		\[ -a d h+b c h+b^2-d^2=0\]
	\[	2 a^2 h-a b n+2 c^2 h+c d n+2 a d+2 b c=0\]
        \[-c h+d g-b=0 \]
\[-a h+b g-d =0\]
\[2 a^2 g-a^2 n+2 c^2 g+c^2 n+4 a c=0\]
\[ -b h n P-2 d h P+a g n-2 c g-c n-2 a=0\]
\[ -d h n P +2 b h P +c g n+2 a g-a n+2 c =0\]
\[-2g h k n P +h^2 n^2P+g^2 n^2+4 h^2 P+4g^2-n^2-4=0\]
\[ h^2 k P+g^2 k-2 g h n-k =0\]
\[-h k P+g n+2 e=0 \]
\[g k-h n+2 f=0\]
\textcolor{red}{It might be unnecessary to list all the equations we have, I'll just put it here for now. There are 21 equations, those might be redundant.} 
\fi

	Notice that $\hN_{12}^3=-1$ if $d=0$ or $2h-k=0$. %\textcolor{red}{Add this argument in appendix, maybe.}
One Diophatine equation we derive is: 	\[2b^2h-b^2k+2d^2h+d^2k=0,\] which can be written as $ \dfrac{b^2}{d^2}=-\dfrac{2h+k}{2h-k}$. So we have $(2h-k)(2h+k)\leq 0$, and since $k>0$, we see that $h\in(-\frac{k}{2},\frac{k}{2})$.  The condition $d_1<14$ implies $n\leq 13$ and $k\leq \sqrt{n^2+4}$, and $k$ is determined by $n$, so we do a brute force search for solutions using parameters $(n,h,k)$. There are 13 cases which pass the non-negative and integral condition of the naive fusion coefficients $\hN_{ij}^k$, which are the cases when $n=1,\ldots, 13$  and $h=-\frac{k}{2}$, for each $k$ corresponding to $n$. In fact, for these cases, the corresponding $\hS$ matrix has the following form:
	$$\begin{pmatrix}1 & \phi_n& 1&\phi_n\\
	\phi_n& -1&\phi_n &-1\\
	1&\phi_n&-1&-\phi_n\\\phi_n&-1&-\phi_n& 1\end{pmatrix}. $$

All the cases can be ruled out by Lemma \ref{semionlemma}  except when $(n,k,h)=(1,1,-\frac{1}{2})$ and $(n,k,h)=(2,2,-1)$. For the first case, we have $a=2 d,b=0,c=d, e=-1,f=0$, and $g=-\frac{1}{2}$. Then $n_{3,3,3} =2d-\frac{1}{2d} $, which is non-negative and integral. Thus $d=-\frac{1}{2}$ or $\frac{1}{2}$. Notice that $d_2=-1$ if $d=-\frac{1}{2}$, which is a contradiction. If $d=\frac{1}{2}$, the corresponding $\hS$-matrix has a modular realization as $\Fib\boxtimes \Sem$. 
For the second case, we have $n_{2,2,2}= d-\frac{1}{d}$. Thus $d=1$ and the corresponding $S$-matrix has a modular realization as $\PSU(2)_6\boxtimes\Sem$.  These are the second pair of $\hS$-matrices.

\end{proof}

\section{Fusion Rules}\label{fusion rules}

Recall that the naive fusion coefficients are defined as $\hN_{ij}^k=N_{ij}^k+N_{ij}^{fk}$, where $i,j,k\in\Pi_0$. To get the fusion coefficients $N_{ij}^k$ for the corresponding super-modular categories, we need to determine how these $\hN_{ij}^k$ split. Note that for the pointed cases, such as the ones in Theorem  \ref{non-selfdual} and Theorem \ref{trivial}, the corresponding super-modular categories split by Proposition \ref{pointedsplits}. Moreover, the $ \hS $ matrices in Theorem \ref{trivial} give the same naive fusion coefficients. From this discussion, we have the following results:
\begin{lem}
	If $ \BB $	is non-self dual super-modular category of rank 8, then $ \BB $ has the same fusion rules as $ \CC(\Z_4,Q)\boxtimes\sVec $ where $\CC(\Z_4,Q)$ is a pointed modular category with $\Z_4$ fusion rules.
\end{lem}

\begin{lem}
	If $ \BB $ is a self-dual super-modular category with Galois group $ G=\langle(0)\rangle $, then $ \BB$ has the same fusion rules as $ \DD\boxtimes \sVec $, where $ \DD $ is a Toric code modular category.
\end{lem}

\begin{lem}
	Let $ \BB $ be a self-dual super-modular category with $ \hS $ of the following form
		$$\begin{pmatrix}1 & 1& 2&\sqrt{6}\\
	1& 1&2 &-\sqrt{6}\\
	2&2&-2&0\\
	\sqrt{6}&-\sqrt{6}&0& 0\end{pmatrix}. $$
	
	Then $ \BB $ has the same fusion rules as the centralizer $\langle f\rangle^\prime$ of either fermion $f$ in the modular category $\SO(12)_2$ (see the Appendix).
\end{lem}
\begin{proof}
$\hN_{11}^1=\hN_{11}^2=\hN_{12}^3=\hN_{22}^3=\hN_{33}^3=0$,  $\hN_{12}^2=\hN_{13}^3=\hN_{22}^2=1$ and $\hN_{23}^3=2$.

We can assume that $N_{22}^2=1$ and $N_{22}^{f2}=0$ by interchanging $X_2$ and $fX_2$ if necessary. Similarly, we assume $N_{13}^3=1$ and $N_{13}^{f3}=0$ by interchanging $X_3$ and $fX_3$ and $X_1$ and $fX_1$ simultaneously, if needed. Using the modified balancing equation on $\hS_{23}$, we get $0=(N_{23}^3-N_{23}^{f3})\theta_3\sqrt{6}$. So we have $N_{23}^3=N_{23}^{f3}=1$. Now we have:
\begin{enumerate}
	\item $f^{\otimes2}=\one$,
	\item $X_1^{\otimes2}=\one$,
	\item $X_2^{\otimes 2}=\one\oplus aX_1\oplus bfX_1\oplus X_2$,
	\item $X_3^{\otimes 2}= \one\oplus X_1\oplus X_2\oplus fX_2$,
	\item $X_1\otimes X_2=aX_2\oplus b fX_2  $,
	\item $X_1\otimes X_3=X_3$,
	\item $X_2\otimes X_3=X_3\oplus fX_3$.
\end{enumerate}

Computing $X_2\otimes  X_2\otimes X_3$ in two ways gives us: $(2+a)X_3\oplus (b+1)fX_3=2X_3\oplus 2fX_3$. So we have $a=0$ and $b=1$.

\end{proof}

\begin{lem}
	Let $ \BB $ be a self-dual super-modular category with
	
	$$\hS=\begin{pmatrix}1 & \phi_1\phi_2& \phi_1&\phi_2\\
	\phi_1\phi_2& 1& -\phi_2  &-\phi_1\\
	\phi_1&-\phi_2&-1&\phi_1\phi_2\\
	\phi_2&-\phi_1&\phi_1\phi_2& -1\end{pmatrix}.$$
Then $ \BB $ has the same fusion rules as $\Fib\boxtimes \PSU(2)_6 $.
\end{lem}
\begin{proof}
The naive fusion coefficients are: $\hN_{11}^1=\hN_{33}^3=2$, $\hN_{11}^2=\hN_{12}^3=\hN_{22}^2=1$,  $\hN_{12}^2=\hN_{13}^3=\hN_{22}^3=\hN_{23}^3=0$. As $\hN_{22}^2=N_{22}^2+N_{22}^{f2}=1$, we assume $N_{22}^2=1$ and  $N_{22}^{f2}=0$ by interchanging $X_2$ and $fX_2$ if necessary.  Then we have $X_2^{\otimes 2}=\one\oplus X_2$, so $X_2$ generates a subcategory $\FF$ with fusion rules like those of $\Fib$, which is necessarily modular. Therefore $\BB\cong \FF\boxtimes\DD$ where $\DD$ is a super-modular category of rank $4$(\cite[Theorem 3.13]{DGNO}). The classification results in \cite{bruillard2017classification} imply that $\BB$ has the same fusion rules as $\Fib\boxtimes\PSU(2)_6$.

\end{proof}

\begin{lem}
	Let $ \BB $ be a self-dual super-modular category with $ \hS $ of the following form
	$$ \begin{pmatrix}1 & d_1& d_2&d_3\\
	d_1&- d_2 &  d_3  &1\\
	d_2& d_3&-1&- d_1\\
	d_3&1 &- d_1 &  d_2
	\end{pmatrix},$$
	where $d_1=1+\sqrt{2}+\sqrt{2+\sqrt{2}}$, $d_2=1+\sqrt{2}+\sqrt{2(2+\sqrt{2})}$ and $d_3=1+\sqrt{2+\sqrt{2}}$.	Then $ \BB $ has the same fusion rules as $ \PSU(2)_{14} $.
\end{lem}
\begin{proof}
	The corresponding naive fusion coefficients are:$\hN_{11}^1=\hN_{11}^3=\hN_{12}^3=\hN_{13}^3=\hN_{33}^3=1$,  $\hN_{11}^2=\hN_{12}^2=\hN_{22}^2=\hN_{22}^3=2$ and $\hN_{23}^3=0$. 
	Since $\hN_{11}^1=N_{11}^1+N_{11}^{f 1}=1$, we can assume $N_{11}^1=1$ and $N_{11}^{f 1}=0$ by interchanging $X_1$ and $f X_1$ if necessary. Similarly, since $\hN_{33}^3=N_{33}^3+N_{33}^{f 3}=1$, we can assume $N_{33}^3=1$ and $N_{33}^{f 3}=0$. Finally, we may use the $X_2$ versus $fX_{2}$ labeling ambiguity to assume that $N_{13}^2=1$. We have:
	
		\begin{enumerate}
		\item $f^{\otimes2}=\one$,
		\item $X_1^{\otimes2}=\one\oplus X_1\oplus aX_2\oplus b fX_2\oplus c X_3\oplus d fX_3$, where $a+b=2$, $c+d=1$,
		\item $X_2^{\otimes2}=\one\oplus g X_1\oplus h  fX_1\oplus lX_2\oplus m fX_2\oplus p X_3\oplus q fX_3$, where $g+h=2$, $l+m=2$ and $p+q=2$,
		\item $X_3^{\otimes2}=\one\oplus r X_1\oplus s   fX_1\oplus  X_3$, where $r+s=1$,
		\item $X_1\otimes X_2= a X_1\oplus  b fX_1\oplus g X_2\oplus h fX_2\oplus  X_3 $,
		\item $X_1\otimes X_3=cX_1\oplus d fX_1\oplus  X_2\oplus rX_3\oplus s fX_3$,
		\item $X_2\otimes X_3=  X_1\oplus p X_2\oplus q fX_2$. 
	\end{enumerate}
	Computing $X_1\otimes X_3\otimes X_3$ in two ways and comparing the coefficients of $X_1$, $fX_1$, $X_2$ and $fX_2$, we have $c+r=2$, $d+s=0$, $ar+bs+1=c+p$ and $br+as=d+q$. Thus we have $c=r=1$, $d=s=0$, $a=p$ and $ b=q$. Applying Corollary \ref{balancing} to $\hS_{23}$, we have
	$|d_1|=|d_1\theta_1 + (p-q)d_2\theta_2|\geq ||(p-q)d_2\theta_2|-d_1|$. If $|p-q|=2$, then $4.26\approx d_1\geq |2d_2-d_1|\approx 5.79$, which is impossible. So we have $p=q=1$. Therefore $a=b=1$. Computing $X_2\otimes X_3\otimes X_3$ in two different ways and comparing the coefficients of $X_2$ and $fX_2$, we have $g=h=1$. Tensoring $X_2\otimes X_2\otimes X_3$ in two ways and comparing the coefficients of $X_1$ and $fX_1$, we have $l=1$ and $m=1$.
\end{proof}
\begin{lem}
	Let $ \BB $ be a self-dual super-modular category with
	$$\hS=\begin{pmatrix}1 & d& 1+d &d^2-1\\
	d&  -(1+d) &-1 &d^2-1\\
	1+d&-1 & d&-(d^2-1)\\
	d^2-1& d^2-1 &- (d^2-1) & 0
	\end{pmatrix},$$
	where $d$ is the largest real root of $x^3 - 3x - 1=0$. Then $ \BB $ has the same fusion rules as $ \PSU(2)_7\boxtimes \sVec $.
\end{lem}
\begin{proof}
We have $\hN_{11}^1=\hN_{11}^2=\hN_{13}^3=0$ and $\hN_{11}^3=\hN_{12}^2=\hN_{12}^3=\hN_{22}^2=\hN_{22}^3=\hN_{23}^3=\hN_{33}^3=1$.

Notice that since $\hN_{22}^2=N_{22}^2+N_{22}^{f 2}=1$, we can assume $N_{22}^2=1$ and $N_{22}^{f 2}=0$ by interchanging $X_2$ and $f X_2$ if necessary. Similarly, we can assume $N_{33}^3=1$, $N_{33}^{f 3}=0$, $\hN_{22}^1=1$ and $\hN_{22}^{f1}=0$. We have  

	\begin{enumerate}
		\item $f^{\otimes2}=\one$,
		\item $X_1^{\otimes2}=\one\oplus aX_3\oplus bfX_3$, where $a+b=1$,
		\item $X_2^{\otimes 2}=\one\oplus X_1\oplus  X_2\oplus gX_3\oplus hfX_3$, where  $g+h=1$,
		\item $X_3^{\otimes 2}= \one\oplus lX_2\oplus mfX_2\oplus X_3 $, where $l+m=1$,
		\item $X_1\otimes X_2=X_2\oplus  pX_3\oplus q fX_3$, where   $p+q=1$,
		\item $X_1\otimes X_3= aX_1\oplus bfX_1\oplus p  X_2\oplus q fX_2$,
		\item $X_2\otimes X_3=pX_1\oplus qfX_1\oplus gX_2+hfX_2\oplus lX_3\oplus mfX_3 $.
	\end{enumerate}  
	
	Computing $X_1\otimes X_1\otimes X_2$ in two different ways and comparing the coefficients of $X_2$ and $fX_2$, we have $ag+bh=1$, $bg+ah=0$. Thus we have $a=g$ and $b=h$. Similarly, comparing the coefficients of $X_3$ and $fX_3$ in $X_1\otimes X_1\otimes X_3$ gives us $a=1$ and $b=0$. Computing $X_2\otimes X_2\otimes X_3$ and comparing the coefficients of $X_3$ and $fX_3$, we have $l=1$ and $m=0$.  Computing $X_1\otimes X_3\otimes X_3$ in two different ways and comparing the coefficients for $X_2$ and $fX_2$, we have $p=1$ and $q=0$.  Observing that the simple objects $\one, X_1,X_2$ and $X_3$ generate a fusion subcategory with the same fusion rules as $\PSU(2)_7$ we obtain the stated result.
\end{proof}

\begin{lem}
	Let $ \BB $ be a self-dual super-modular category. Suppose that the corresponding $ \hS $ has one  of the following forms
$$\begin{pmatrix}1 & \phi_1^2& \phi_1&\phi_1\\
\phi_1^2& 1& -\phi_1  &-\phi_1\\
\phi_1&-\phi_1&-1&\phi_1^2\\
\phi_1&-\phi_1&\phi_1^2& -1\end{pmatrix},\quad\begin{pmatrix}1 &  \phi_2^2&  \phi_2& \phi_2\\
\phi_2^2& 1& - \phi_2  &- \phi_2\\
\phi_2&- \phi_2&-1& \phi_2^2\\
\phi_2&- \phi_2& \phi_2^2& -1\end{pmatrix}$$ 

$$\begin{pmatrix}1 &   \phi_1& 1 & \phi_1\\
\phi_1& -1& \phi_1  &-1 \\
1& \phi_1&-1&  -\phi_1\\
\phi_1&- 1&  -\phi_1& 1\end{pmatrix}, \quad\begin{pmatrix}1 &    \phi_2& 1 &  \phi_2\\
\phi_2& -1&   \phi_2  &-1 \\
1& \phi_2&-1&   -\phi_2\\
\phi_2&- 1&   -\phi_2& 1\end{pmatrix}, $$
	then $ \BB $ has the same fusion rules as  $\Fib\boxtimes\Fib\boxtimes\sVec$, $[\PSU(2)_6 \boxtimes\PSU(2)_6]_{\Z_2}$, $\Sem\boxtimes\PSU(2)_6\boxtimes\sVec$, or $\Sem\boxtimes\Fib\boxtimes\sVec$, respectively.
\end{lem}
\begin{proof}
	
Consider the first $\hS$-matrix. We have $\hN_{11}^1=\hN_{11}^2=\hN_{11}^3=\hN_{12}^3=\hN_{22}^2=\hN_{33}^3=1$ and $\hN_{12}^2=\hN_{13}^3=\hN_{22}^3=\hN_{23}^3=0$. 	\iffalse Assume $N_{11}^1=1$ and $N_{11}^{f1}=0$ by interchanging $X_1$ and $fX_1$ if necessary.\fi Without loss of generality, we may assume $N_{22}^2=1$, $N_{22}^{f2}=0$ by interchanging $X_2$ and $fX_2$ if necessary. Thus $X_2^{\otimes 2}=\one\oplus X_2$, so $X_2$ generates a subcategory $\FF$ with fusion rules like those of $\Fib$, which is necessarily modular. In particular $\BB\cong \FF\boxtimes\DD$, where $\DD$ is a super-modular category of rank $4$.  The classification results of \cite{bruillard2017classification} now imply that $\BB$ has the same fusion rules as $\Fib\boxtimes\Fib\boxtimes\sVec$.

\iffalse
$N_{33}^3=1$ and $N_{33}^{f3}=0$. We have 
\begin{enumerate}
	\item $f^{\otimes2}=\one$
	\item $X_1^{\otimes2}=\one\oplus X_1\oplus aX_2\oplus bfX_2\oplus cX_3\oplus dfX_3$.
	\item $X_2^{\otimes 2}=\one\oplus X_2$
	\item $X_3^{\otimes 2}= \one\oplus X_3 $
	\item $X_1\otimes X_2= aX_1\oplus bfX_1\oplus sX_3\oplus t fX_3$.   
	\item $X_1\otimes X_3= cX_1\oplus dfX_1\oplus s X_2\oplus t fX_2$. 
	\item $X_2\otimes X_3= sX_1\oplus t fX_1 $.
\end{enumerate}
where $a+b=c+d=s+t=1$. \\
Computing $X_1\otimes X_2\otimes X_2$ in two ways, we get
\[(1+a)X_1\oplus bfX_1\oplus sX_3\oplus tfX_3=2X_1\oplus(bt+st)fX_1\oplus(as+bt)X_3\oplus(bs+at)fX_3\]
This gives $a=1$ and $b=0$.
Similarly, computing $X_1\otimes X_3\otimes X_3$ gives us 
$$(1+c)X_1\oplus dfX_1\oplus sX_2\oplus t fX_2=2X_1\oplus (cs+td)X_2\oplus(ct+sd)fX_2$$
Thus $c=1$ and $d=0$.
Finally, we compute $X_1\otimes X_1\otimes X_2$ and get 
$$\one\oplus (1+s) X_1\oplus tfX_1\oplus 2X_2\oplus sX_3\oplus tfX_3=\one\oplus (1+s)X_1\oplus t fX_1\oplus 2X_2\oplus X_3$$
So we have $s=1$ and $t=0$. The fermionic quotient has the same fusion rules of Fib $\boxtimes$ Fib. \\

\fi

For the second $\hS$-matrix, we have that the associated naive fusion coefficients are $\hN_{11}^1=4$, $\hN_{11}^2=\hN_{11}^3=\hN_{22}^2=\hN_{33}^3=2$, $\hN_{12}^3=1 $, $\hN_{12}^2=\hN_{13}^3=\hN_{22}^3=\hN_{23}^3=0$.
We may assume $N_{12}^3=1$ and $N_{12}^{f3}=0$ by interchanging $X_3$ and $fX_3$ if necessary. Using Corollary \ref{balancing} on $\hS_{12}$ gives
$$-\theta_1\theta_2\phi_2=(N_{12}^1-N_{12}^{f1})\phi_2^2\theta_1+\phi_2\theta_3.$$
Dividing by $\phi_2$, we have 
$$-\theta_1\theta_2=(N_{12}^1-N_{12}^{f1})\phi_2\theta_1+\theta_3.$$
Taking absolute value on both sides, we get
$$1=\bigl|(N_{12}^1-N_{12}^{f1})\phi_2\theta_1+\theta_3\bigr|\ge \bigl||(N_{12}^1-N_{12}^{f1})\phi_2|-1\bigr|.$$
So we must have $N_{12}^1=N_{12}^{f1}=1$. Similarly, applying Corollary \ref{balancing} to $\hS_{33}$ and $\hS_{13}$ gives\\
$$-\theta_3^2=1+(N_{33}^3-N_{33}^{f3})\phi_2\theta_3,\quad -\theta_1\theta_3\phi_2=(N_{13}^1-N_{13}^{f1})\phi_2^2+\phi_2\theta_2$$
and we get $N_{33}^3=N_{33}^{f3}=1$ and $N_{11}^3=N_{11}^{f3}=1$.  A parallel calculation for $\hS_{22}$ yields $N_{22}^2=N_{22}^{f2}=1$.
By using Corollary \ref{balancing} again for $\hS_{11}$, we get 	$$\theta_1^2=(N_{11}^1-N_{11}^{f1})\phi_2^2\theta_1+1.$$
The potential choices of $(N_{11}^1,N_{11}^{f1})$ are $(2,2)$, $(4,0)$, $(0,4)$, $(1,3)$ and $(3,1)$, but since $\phi_2^2>2$ the only possibility is $(2,2)$. This category has the same fusion rules as $[\PSU(2)_6\boxtimes\PSU(2)_6]_{\Z_2}$, see the Appendix.

\iffalse
Consider the third $\hS$ matrix, we have $ \hN_{11}^2= \hN_{11}^3=\hN_{12}^2=\hN_{22}^2=\hN_{22}^3=\hN_{23}^3=\hN_{33}^3=0$ and $ \hN_{11}^1=\hN_{12}^3=\hN_{13}^3=1$. Assume $N_{11}^1=1$ and $N_{11}^{f1}=0$ by interchanging $X_1$ and $fX_1$ if necessary.  Thus $X_1^{\ot 2}=\one\oplus X_1$, so that $X_1$ generates a modular subcategory with the same fusion rules as $\Fib$. By Lemma \ref{semionlemma} we also have a pointed modular subcategory with fusion rules like $\Sem$, thus we obtain the fusion rules of $\Fib\boxtimes\Sem\boxtimes\sVec$.

	The last $\hS$ matrix has $  \hN_{11}^2= \hN_{11}^3=\hN_{12}^2=\hN_{22}^2=\hN_{22}^3=\hN_{23}^3=\hN_{33}^3=0$, $ \hN_{12}^3=1 $ and  $ \hN_{11}^1=\hN_{13}^3=2 $. Assume $N_{12}^3=1$ and $N_{12}^{f3}=0$ by interchanging $X_3$ and $fX_3$ if necessary.  Using Lemma \ref{balancing} on $\hS_{11}$, we have 
$$-\theta_1^2=1+(N_{11}^1-N_{11}^{f1})\theta_1\phi_2$$
By applying absolute value both sides and triangle inequality, we get $N_{11}^1=N_{11}^{f1}=1$. 
Similarly, we can get $N_{13}^3=N_{13}^{f3}=1$.  The fermionic quotient has the same fusion rules as $\Sem\boxtimes\PSU(2)_6$.
\fi
In the last two cases, observe that $\BB$ must contain a modular subcategory of the form $\CC(\Z_2,Q)$ by Lemma \ref{semionlemma}. Then $\BB\cong\CC(\Z_2,Q)\boxtimes\DD$, where $\DD$ is a rank $4$ super-modular category.  The result now follows from the classification in \cite{bruillard2017classification}.

\end{proof}

\iffalse
\er{the following subsection is maybe superfluous.  the conjecture is already published, the twist equation statement doesn't add much (except to say we can't do something), and the gauss sum is a future direction.  I have commented it out for now}

\subsection{Observations and Questions}
The following conjecture was first asked in \cite{bruillard2017classification} when classifying super-modular categories up to rank 6. From our results so far, it still holds.
	\begin{conj}
		If $ \BB $ is a super-modular category and the corresponding  $\hS  $ is the $ S $-matrix for some modular category. Then $ \BB $ is split super-modular.
	\end{conj}
 The twisted equation(see \cite[Theorem 2.14]{RSW}) for modular categories is very helpful for classifying low modular categories. One can use this property to figure out the bound for the dimensions, thus could bound all the solutions. We do not have a parallel inequality for the fermionic quotient of super-modular categories so far. By mimicking the proof, we can have some inequalities for the dimensions of super-modular categories. However, these conditions are too weak to use for our classification process. The major obstacle here is that the $ S $-matrix is degenerate. 

 The Gauss sum mentioned in Section 3.4 is an invariant for pre-modular categories. For super-modular, it is always 0. Higher Gauss sum was recently studied in \cite{HigherGauss}. It is an interesting question to study the higher Gauss sums for super-modular categories.
\fi

\section*{Appendix}
Here we record the data for some of the realizations of the super-modular categories that appear in this article, both modular and super-modular, as well as the families of categories in which they reside.  We write the $T$-matrix as an $n$-tuple with the understanding that these are the diagonal entries.

\subsection{Pointed Modular Categories}
Pointed braided fusion categories are classified, see \cite{DGNO}.  They correspond to pairs $(A,Q)$, where $Q$ is a  symmetric quadratic form on $A$ (with values in $U(1)$).  The fusion rules of $\CC(A,Q)$ are the same as the multiplication in $A$, and the $S$- and $T$-matrices are determined by $Q$ as follows: $S_{a,b}=\frac{Q(a+b)}{Q(a)Q(b)}$ and $\theta_a=Q(a)$. If the symmetric bilinear form given by $S_{a,b}$ is non-degenerate then $\CC(A,Q)$ is modular. 

For example the semion theory $\Sem=\CC(\Z_2,Q)$ that appears in our classification has the following modular data:
$S=\begin{pmatrix}
1 & 1\\1 & -1
\end{pmatrix}$, and $T=(1,i)$.

\subsection{$\PSU(2)_k$}
The rank $k+1$ modular category $\SU(2)_k$ obtained from $U_q\mathfrak{sl}_2$ at $q=e^{\pi i/(2+k)}$ contains the subcategory $\PSU(2)_k$ whose simple objects have even labels (``integer spin" in the physics literature). Denote by $\varpi$ the fundamental weight of type $A_1$, so that $X_\varpi$ tensor generates $\SU(2)_k$. The object labeled by $\frac{k}{2}\varpi$ is always invertible.  When $k\equiv 2\pmod{4}$ the category $\PSU(2)_k$ is super-modular with $f=X_{\frac{k}{2}\varpi}$, when $4\mid k$, there is a boson $b=X_{\frac{k}{2}\varpi}$ in $\PSU(2)_k$, and when $k$ is odd, $\PSU(2)_k$ is modular, with $X_{\frac{k}{2}\varpi}$ a semion (not in $\PSU(2)_k$.)

The (modular) Fibonacci theory Fib $=\PSU(2)_3^{rev}$ as well as $\PSU(2)_7$ appear in our classification, and the data can be found in \cite{RSW}.

Some low rank super-modular categories that appear in this article are:
\begin{itemize}
    \item $\PSU(2)_6$ with data:

$S=\begin{pmatrix}
1 & 1+\sqrt{2} \\
1+\sqrt{2} & -1 
\end{pmatrix}\ot \begin{pmatrix} 1 & 1\\ 1 & 1\end{pmatrix}$ and $T=(1,i)\ot (1,-1)$.  
\item $\PSU(2)_{10}$ with data:

$S=\begin{pmatrix}
1 & 2+\sqrt{3} & 1+\sqrt{3}\\
2+\sqrt{3} & 1 & -1-\sqrt{3}\\
1+\sqrt{3} & -1-\sqrt{3} & 1+\sqrt{3}
\end{pmatrix}\ot \begin{pmatrix} 1 & 1\\ 1 & 1\end{pmatrix}$ and $T=(1,-1,e^{\pi i/3})\ot (1,-1)$. 
\item $\PSU(2)_{14}$ with data: 

$S=\begin{pmatrix}
1 & 1+x & 1+\sqrt{2}+x & 1+\sqrt{2} + \sqrt{2}x\\
1+x & 1+\sqrt{2} + \sqrt{2}x & 1 & -1-\sqrt{2} -x\\
1+\sqrt{2}+x & 1 & -1-\sqrt{2}-\sqrt{2}x & 1+x\\
1+\sqrt{2}+\sqrt{2}x & -1-\sqrt{2} -x & 1+x & -1
\end{pmatrix}\ot \begin{pmatrix} 1 & 1\\ 1 & 1\end{pmatrix}$, where
$x=\sqrt{2+\sqrt{2}}$ and $T=(1, e^{i\pi/4},e^{3i\pi/4},-i)\ot (1,-1)$.
\end{itemize}

The full sequence of super-modular categories $\PSU(2)_{4m+2}$ was studied in \cite{16fold,bonderson2018congruence}, where the modular data can be found.  If we order the simple objects $[\one,X_1,\ldots,X_{r-1},fX_{r-1},\ldots,fX_1,f]=[Y_0,\ldots,Y_{2(r-1)}]$ the fusion rules are completely determined by the rule $Y_1\ot Y_k\cong Y_{k-1}\oplus Y_k\oplus Y_{k+1}$ for $0<k<2(r-1)$ and the obvious rules involving $Y_{2(r-1)}=f$ and $Y_0=\one$.

\subsection{Other examples}
The following are spin modular categories coming from quantum groups with fermion $f$ so that the subcategory $\langle f\rangle^\prime$ is super-modular, where $r,m\in\mathbb{N}$:
\begin{itemize}
\item $\SU(4k+2)_{4m+2}$, 
\item $\SO(2k+1)_{2m+1}$, 
\item $Sp(2r)_m$ with $rm=2 \pmod{4}$,
\item  $\SO(2r)_m$ with $r=2 \pmod{4}$ and $m=2 \pmod{4}$,
\item  $(E_7)_{4m+2}$.
\end{itemize}
The pointed sub-category of the rank $13$ modular category $\SO(12)_2$ is $\sVec\boxtimes\sVec$ and hence contains two fermions labeled by $2\varpi_5$ and $2\varpi_6$, where $\varpi_i$ are the fundamental weights of type $D_6$.  The centralizer of either of these fermions is super-modular and has modular data:

$S:=\begin{pmatrix}1 & 1& 2&\sqrt{6}\\
	1& 1&2 &-\sqrt{6}\\
	2&2&-2&0\\
	\sqrt{6}&-\sqrt{6}&0& 0\end{pmatrix}\otimes \begin{pmatrix} 1 & 1\\ 1 & 1\end{pmatrix} $ and $T=(1,1,e^{2\pi i/3},e^{3\pi i/8})\ot (1,-1).$
	If we label the simple objects of dimension $\sqrt{6}$ by $X_3$ and $fX_3$ then the fusion rules are determined by $X_3^{\ot 2}\cong\one\oplus X_1\oplus X_2\oplus fX_2$, $X_1^{\ot 2}\cong\one$, $X_2^{\ot 2}\cong\one\oplus fX_1\oplus X_2$ and $X_2\ot X_3\cong X_3\oplus fX_3$.
	
Finally we observe that if $(\CC_1,f_1)$ and $(\CC_2,f_2)$ are spin modular categories, then $(f_1,f_2)\in\CC_1\boxtimes\CC_2$ is a boson and hence can be condensed to obtain a new spin modular category $([\CC_1\boxtimes\CC_2]_{\Z_2})_0$, where we de-equivariantize by $\Rep(\Z_2)\cong \langle (f_1,f_2)\rangle$ and then take the trivial component of the corresponding $\Z_2$-grading.  For example applying this to $\PSU(2)_6$ we obtain the prime rank $8$ example $(\PSU(2)_6\boxtimes\PSU(2)_6)_{\Z_2}$ with data:

$S:=\begin{pmatrix}1 &  3+2\sqrt{2}&  1+\sqrt{2}& 1+\sqrt{2}\\
	3+2\sqrt{2}& 1& - 1-\sqrt{2}  &- 1-\sqrt{2}\\
	1+\sqrt{2}&- 1-\sqrt{2}&-1& 3+2\sqrt{2}\\
	1+\sqrt{2}&- 1-\sqrt{2}& 3+2\sqrt{2}& -1\end{pmatrix}\otimes \begin{pmatrix} 1 & 1\\ 1 & 1\end{pmatrix}$ and $T=(1,-1,i,i)\ot (1,-1).$
	The fusion rules may be readily determined from those of $\PSU(2)_6$ by condensing the boson $b:=(f_1,f_1)$.  Notice that $b\ot X\not\cong X$ for any simple $X$ so that there is no ambiguity in labeling the objects in the de-equivariantization.  Setting $f:=[(f_1,\one)]=[(\one,f_1)]$ we have
	
	\begin{align*}
	   & X_1^{\ot 2}\cong \one\oplus 2(X_1\oplus fX_1)\oplus X_2\oplus fX_2\oplus X_3\oplus fX_3, \quad X_1\ot X_2\cong X_3\oplus X_1\oplus fX_1\\ &X_1\ot X_3\cong X_2\oplus X_3\oplus fX_3,\quad X_2\ot X_3\cong X_1,\; \mathrm{ and }\; X_2^{\ot 2}\cong \one\oplus X_2\oplus fX_2
	\end{align*} from which all fusion rules can be recovered.
\bibliographystyle{plain}
\bibliography{Reference.bib}

\begin{thebibliography}{10}

\bibitem{AMP}
Narjess Afzaly, Scott Morrison, and David Penneys.
\newblock The classification of subfactors with index at most $5 \frac{1}{4}$.
\newblock {\em arXiv preprint arXiv:1509.00038}, 2015.

\bibitem{bakalov2001lectures}
Bojko Bakalov and Alexander~A Kirillov~{Jr.}
\newblock {\em Lectures on tensor categories and modular functors}, volume~21.
\newblock American Mathematical Soc., 2001.

\bibitem{Bl}
C.~Blanchet.
\newblock A spin decomposition of the {V}erlinde formulas for type {A} modular
  categories.
\newblock {\em Comm. Math. Phys.}, 257(1):1--28, 2005.

\bibitem{BM}
C.~Blanchet and G.~Masbaum.
\newblock Topological quantum field theories for surfaces with spin structure.
\newblock {\em Duke Math. J.}, 82(2):229--268, 1996.

\bibitem{BCT}
Parsa Bonderson, Meng Cheng, and Alan Tran.
\newblock Fermionic topological phases and modular transformations.
\newblock 2017.

\bibitem{bonderson2018congruence}
Parsa Bonderson, Eric Rowell, Zhenghan Wang, and Qing Zhang.
\newblock Congruence subgroups and super-modular categories.
\newblock {\em Pacific J. Math.}, 296(2):257--270, 2018.

\bibitem{Bon}
P.H. Bonderson.
\newblock {\em Non-{A}belian anyons and interferometry}.
\newblock PhD thesis, California Institute of Technology, 2007.

\bibitem{Brug}
A.~Brugui\`eres.
\newblock Cat\'egories pr\'emodulaires, modularisations et invariants des
  vari\'et\'es de dimension 3.
\newblock {\em Math. Ann.}, 316(2):215--236, 2000.

\bibitem{MR3548123}
Paul Bruillard.
\newblock Rank 4 premodular categories.
\newblock {\em New York J. Math.}, 22:775--800, 2016.
\newblock With an Appendix by C\'{e}sar Galindo, Siu-Hung Ng, Julia Plavnik,
  Eric Rowell and Zhenghan Wang.

\bibitem{16fold}
Paul Bruillard, Cesar Galindo, Tobias Hagge, Siu-Hung Ng, Julia~Yael Plavnik,
  Eric~C Rowell, and Zhenghan Wang.
\newblock Fermionic modular categories and the 16-fold way.
\newblock {\em arXiv preprint arXiv:1603.09294}, 2016.

\bibitem{bruillard2013classification}
Paul Bruillard, C{\'e}sar Galindo, Seung-Moon Hong, Yevgenia Kashina, Deepak
  Naidu, Sonia Natale, Julia~Yael Plavnik, Eric Rowell, et~al.
\newblock Classification of integral modular categories of frobenius-perron
  dimension pq\^{} 4 and p\^{} 2q\^{} 2.
\newblock {\em Canad. Math. Bull.}, 57(4):721--734, 2014.

\bibitem{bruillard2016classification}
Paul Bruillard, C{\'e}sar Galindo, Siu-Hung Ng, Julia~Y Plavnik, Eric~C Rowell,
  and Zhenghan Wang.
\newblock On the classification of weakly integral modular categories.
\newblock {\em J. Pure Appl. Algebra}, 220(6):2364--2388, 2016.

\bibitem{bruillard2017classification}
Paul Bruillard, C{\'e}sar Galindo, Siu-Hung Ng, Julia~Yael Plavnik, Eric~C
  Rowell, and Zhenghan Wang.
\newblock Classification of super-modular categories by rank.
\newblock {\em arXiv preprint arXiv:1705.05293}, 2017.

\bibitem{bruillard2017categorical}
Paul Bruillard, Paul Gustafson, Julia~Yael Plavnik, and Eric~Carson Rowell.
\newblock Categorical dimension as a quantum statistic and applications.
\newblock {\em arXiv preprint arXiv:1710.10284}, 2017.

\bibitem{BNRW1}
Paul Bruillard, Siu-Hung Ng, Eric Rowell, and Zhenghan Wang.
\newblock Rank-finiteness for modular categories.
\newblock {\em J. Amer. Math. Soc.}, 29(3):857--881, 2016.

\bibitem{BNRW2}
Paul Bruillard, Siu-Hung Ng, Eric~C Rowell, and Zhenghan Wang.
\newblock On classification of modular categories by rank.
\newblock {\em Int. Math. Res. Not.}, 2016(24):7546--7588, 2016.

\bibitem{BOM}
Paul Bruillard and Carlos~M. Ortiz-Marrero.
\newblock Classification of rank 5 premodular categories.
\newblock {\em J. Math. Phys.}, 59(1):011702, 8, 2018.

\bibitem{Creamer}
Daniel Creamer.
\newblock {\em A computational approach to classifying low rank modular tensor
  categories}.
\newblock PhD thesis, Texas A\&M University, 2018.

\bibitem{DMNO}
A.~Davydov, M.~M\"{u}ger, D.~Nikshych, and V.~Ostrik.
\newblock The {W}itt group of non-degenerate braided fusion categories.
\newblock {\em J. Reine Angew. Math.}, 677:135--177, 2013.

\bibitem{davydov2013structure}
Alexei Davydov, Dmitri Nikshych, and Victor Ostrik.
\newblock On the structure of the witt group of braided fusion categories.
\newblock {\em Selecta Math.}, 19(1):237--269, 2013.

\bibitem{Del}
P.~Deligne.
\newblock Cat\'egories tannakiennes.
\newblock {\em The {G}rothendieck Festschrift, Vol. {II}. Progr. Math.},
  87:111--195, 1990.

\bibitem{DGNO}
Vladimir Drinfeld, Shlomo Gelaki, Dmitri Nikshych, and Victor Ostrik.
\newblock On braided fusion categories i.
\newblock {\em Selecta Math.}, 16(1):1--119, 2010.

\bibitem{EM1}
Cain Edie-Michell.
\newblock Classifying fusion categories $\otimes$-generated by an object of
  small {F}robenius-{P}erron dimension.
\newblock {\em {Preprint} arxiv:1810.05717}, 2018.

\bibitem{EM2}
Cain Edie-Michell.
\newblock A complete classification of pivotal fusion categories
  $\otimes$-generated by an object of dimension $\frac{1+\sqrt{5}}{2}$.
\newblock {\em {Preprint} arXiv:1904.08909}, 2019.

\bibitem{MR1617921}
Pavel Etingof and Shlomo Gelaki.
\newblock Some properties of finite-dimensional semisimple {H}opf algebras.
\newblock {\em Math. Res. Lett.}, 5(1-2):191--197, 1998.

\bibitem{MR3242743}
Pavel Etingof, Shlomo Gelaki, Dmitri Nikshych, and Victor Ostrik.
\newblock {\em Tensor categories}, volume 205 of {\em Mathematical Surveys and
  Monographs}.
\newblock American Mathematical Society, Providence, RI, 2015.

\bibitem{MR2098028}
Pavel Etingof, Shlomo Gelaki, and Viktor Ostrik.
\newblock Classification of fusion categories of dimension {$pq$}.
\newblock {\em Int. Math. Res. Not.}, (57):3041--3056, 2004.

\bibitem{frohlich1993algebraic}
Albrecht Fr{\"o}hlich, Martin~J Taylor, and Martin~J Taylor.
\newblock {\em Algebraic number theory}, volume~27.
\newblock Cambridge University Press, 1993.

\bibitem{GV}
C{\'e}sar Galindo and C{\'e}sar~F. Venegas-Ram{\'i}rez.
\newblock Categorical fermionic actions and minimal modular extensions.
\newblock {\em arXiv preprint arXiv:1712.07097}, 2017.

\bibitem{gepnerkapsutin}
Doron Gepner and Anton Kapustin.
\newblock On the classification of fusion rings.
\newblock {\em Phys. Lett. B}, 349(1-2):71--75, 1995.

\bibitem{Green}
David Green.
\newblock Classification of rank 6 modular categories with galois group
  $\langle (012)(345)\rangle$.
\newblock {\em arXiv preprint arXiv:1908.07128}, 2019.

\bibitem{jones2019rank}
Corey Jones, Scott Morrison, Dmitri Nikshych, and Eric~C Rowell.
\newblock Rank-finiteness for g-crossed braided fusion categories.
\newblock {\em {Preprint} arXiv:1902.06165}, 2019.

\bibitem{KLW}
Tian Lan, Liang Kong, and Xiao-Gang Wen.
\newblock Theory of (2+1)-dimensional fermionic topological orders and
  fermionic/bosonic topological orders with symmetries.
\newblock {\em Phys. Rev. B}, 94:155113, Oct 2016.

\bibitem{mahler1935grossten}
Kurt Mahler.
\newblock {\em {\"U}ber den gr{\"o}ssten Primteiler spezieller Polynome zweiten
  Grades}.
\newblock Johansen, 1935.

\bibitem{MP}
Scott Morrison and David Penneys.
\newblock Monoidal categories enriched in braided monoidal categories.
\newblock {\em Int. Math. Res. Not. IMRN}, (11):3527--3579, 2019.

\bibitem{M3}
M.~M\"{u}ger.
\newblock Galois extensions of braided tensor categories and braided crossed
  {G}-categories.
\newblock {\em J. Algebra}, 277:256--281, 2004.

\bibitem{MR1749250}
Michael M\"{u}ger.
\newblock Galois theory for braided tensor categories and the modular closure.
\newblock {\em Adv. Math.}, 150(2):151--201, 2000.

\bibitem{Mu03}
Michael M{\"u}ger.
\newblock On the structure of modular categories.
\newblock {\em Proc. London Math. Soc. (3)}, 87(2):291--308, 2003.

\bibitem{naidu2011finiteness}
Deepak Naidu and Eric~C Rowell.
\newblock A finiteness property for braided fusion categories.
\newblock {\em Algebr. Represent. Theory}, 14(5):837--855, 2011.

\bibitem{Nayaketal}
Chetan Nayak, Steven~H. Simon, Ady Stern, Michael Freedman, and Sankar
  Das~Sarma.
\newblock Non-abelian anyons and topological quantum computation.
\newblock {\em Rev. Modern Phys.}, 80(3):1083--1159, 2008.

\bibitem{ng2010congruence}
Siu-Hung Ng and Peter Schauenburg.
\newblock Congruence subgroups and generalized frobenius-schur indicators.
\newblock {\em Comm. Math. Phys.}, 300(1):1--46, 2010.

\bibitem{ostrik2002fusion}
Viktor Ostrik.
\newblock Fusion categories of rank 2.
\newblock {\em Math. Res. Lett.}, 10(2-3):177--183, 2003.

\bibitem{RSW}
Eric Rowell, Richard Stong, and Zhenghan Wang.
\newblock On classification of modular tensor categories.
\newblock {\em Comm. Math. Phys.}, 292(2):343--389, 2009.

\bibitem{Sawin}
S.~Sawin.
\newblock Invariants of spin three-manifolds from {C}hern-{S}imons theory and
  finite-dimensional {H}opf algebras.
\newblock {\em Adv. Math.}, 165(1):35--70, 2002.

\bibitem{Usher}
Robert Usher.
\newblock Fermionic {$6j$}-symbols in superfusion categories.
\newblock {\em J. Algebra}, 503:453--473, 2018.

\bibitem{WW}
Kevin Walker and Zhenghan Wang.
\newblock (3+1)-{TQFT}s and topological insulators.
\newblock {\em Front. Phys.}, 7(2):150--159, 2012.

\bibitem{XGWen1506.05768}
Xiao-Gang Wen.
\newblock {A theory of 2+1D bosonic topological orders}.
\newblock {\em National Science Review}, 3(1):68--106, 11 2015.

\bibitem{yu}
Zhiqiang Yu.
\newblock On slightly degenerate fusion categories.
\newblock {\em arXiv preprint arXiv:1903.06345}, 2019.

\end{thebibliography}

\end{document}